\documentclass{amsart}

\usepackage[a4paper,hmargin=3.45cm,vmargin=4cm]{geometry}
\usepackage{amsfonts,amssymb,amscd,amstext}
\usepackage[utf8]{inputenc}
\usepackage[colorlinks=true,linkcolor=blue,citecolor=red]{hyperref}

\renewcommand{\leq}{\leqslant}
\renewcommand{\geq}{\geqslant}
\newcommand{\ptl}{\partial}
\newcommand{\rr}{{\mathbb{R}}}
\newcommand{\rrn}{\mathbb{R}^{n+1}}
\newcommand{\la}{\lambda}
\newcommand{\sph}{\mathbb{S}}
\newcommand{\nn}{\mathbb{N}}
\newcommand{\sub}{\subset}
\newcommand{\subeq}{\subseteq}
\newcommand{\escpr}[1]{\big<#1\big>}
\newcommand{\Sg}{\Sigma} \newcommand{\sg}{\sigma}
\newcommand{\Om}{\Omega}
\newcommand{\barOm}{\overline{\Omega}}

\newcommand{\var}{\varphi}
\newcommand{\ric}{\text{Ric}}
\newcommand{\ind}{\mathcal{Q}}
\newcommand{\indo}{\mathcal{I}}
\newcommand{\h}{\mathcal{H}}
\newcommand{\ga}{\gamma}

\DeclareMathOperator{\divv}{div}

\newtheorem{theorem}{Theorem}[section]
\newtheorem{proposition}[theorem]{Proposition}
\newtheorem{lemma}[theorem]{Lemma}
\newtheorem{corollary}[theorem]{Corollary}

\theoremstyle{definition}

\newtheorem{remark}[theorem]{Remark}
\newtheorem{remarks}[theorem]{Remarks}
\newtheorem{example}[theorem]{Example}

\theoremstyle{remark}

\numberwithin{equation}{section}

\begin{document}

\bibliographystyle{amsplain}  

\title[Isoperimetry and stability for perturbations of Gaussian measures]{Isoperimetric and stable sets for log-concave perturbations of  Gaussian measures}

\author[C.~Rosales]{C\'esar Rosales}
\address{Departamento de Geometr\'{\i}a y Topolog\'{\i}a \\
Universidad de Granada \\ E--18071 Granada \\ Spain}
\email{crosales@ugr.es}

\date{\today}

\thanks{The author is supported by MICINN-FEDER grant MTM2010-21206-C02-01, Junta de Andaluc\'ia grants FQM-325 and P09-FQM-5088, and the grant PYR-2014-23 of the GENIL program of CEI BioTic GRANADA} 



\begin{abstract}
Let $\Om$ be an open half-space or slab in $\rrn$ endowed with a perturbation of the Gaussian measure of the form $f(p):=\exp(\omega(p)-c|p|^2)$, where $c>0$ and $\omega$ is a smooth concave function depending only on the signed distance from the linear hyperplane parallel to $\ptl\Om$. In this work we follow a variational approach to show that half-spaces perpendicular to $\ptl\Om$ uniquely minimize the weighted perimeter in $\Om$ among sets enclosing the same weighted volume. The main ingredient of the proof is the characterization of half-spaces parallel or perpendicular to $\ptl\Om$ as the unique stable sets with small singular set and null weighted capacity. Our methods also apply for $\Om=\rrn$, which produces in particular the classification of stable sets in Gauss space and a new proof of the Gaussian isoperimetric inequality. Finally, we use optimal transport to study the weighted minimizers when the perturbation term $\omega$ is concave and possibly non-smooth.  
\end{abstract}

\maketitle

\thispagestyle{empty}

\section{Introduction}
\label{sec:intro}
\setcounter{equation}{0}

The \emph{Gaussian isoperimetric inequality} sta\-tes that in Euclidean space $\rrn$ endowed with the weight $\ga_c(p):=\exp(-c |p|^2)$, $c>0$, any half-space is an isoperimetric region, i.e., it has the least weighted perimeter among sets enclosing the same weighted volume. This result was independently proved in the mid-seventies by Sudakov and Tirel'son \cite{st}, and Borell~\cite{borell}, by means of an approximation argument which is sometimes attributed to Poincar\'e. As it is explained in \cite[Thm.~18.2]{gmt} and \cite[Prop.~6]{rosisoperimetric}, the Gaussian probability measure in $\rrn$ is the limit of orthogonal projections into $\rrn$ of high dimensional spheres $\sph^{m}(\sqrt{m+1})$ with uniform probability densities. Hence, the isoperimetric inequality in Gauss space can be deduced as the limit of the spherical isoperimetric inequality. The complete characterization of equality cases does not follow from the previous approach and was subsequently studied by Carlen and Kerce~\cite{ck}. Indeed, they discussed when equality holds in a more general functional inequality of Bobkov, and showed in particular that any isoperimetric set in the Gauss space coincides, up to a set of volume zero, with a half-space. 

Other proofs of the Gaussian isoperimetric inequality were found by Ehrhard~\cite{ehrhard}, who adapted Steiner symmetrization to produce a Brunn-Minkowski inequality in the Gaussian context, by Bakry and Ledoux~\cite{bl}, see also \cite{ledoux-gaussian}, who gave a semigroup proof involving Ornstein-Uhlenbeck operators, and by Bobkov~\cite{bobkov2}, as a consequence of sharp two point functional inequalities and the central limit theorem. Bobkov's inequalities were later extended to the sphere and used to deduce isoperimetric estimates for the unit cube by Barthe and Maurey~\cite{bm}. In \cite{morgandensity}, Morgan employed a Heintze-Karcher type inequality to provide a variational proof, which also includes the uniqueness of isoperimetric regions. More recently, Cianchi, Fusco, Maggi and Pratelli~\cite{cfmp}, and Barchiesi, Brancolini and Julin~\cite{bbj} obtained sharp quantitative estimates for the Gaussian isoperimetric inequality. The paper \cite{cfmp} also characterizes the equality cases after a careful study of Ehrhard symmetrization. An interesting result of Bobkov and Udre~\cite{bobkov-udre} shows that a symmetric probability measure $\mu$ on $\rr$ such that all coordinate half-spaces in $\rrn$ are isoperimetric sets for the product measure $\mu^{n+1}$ is necessarily of Gaussian type.

The mathematical interest in the Gaussian isoperimetric question comes also from its wide range of applications in Probability Theory and Functional Analysis. The surveys of Ledoux~\cite{ledoux-survey}, \cite{markov} show its relation to the concentration of measures phenomenon, the theory of spectral gaps for diffusion generators, and the logarithmic Sobolev inequalities. Morgan's book~\cite[Ch.~18]{gmt} includes a brief description of how the Gaussian logarithmic Sobolev inequality was used by Perelman in his proof of the Poincar\'e conjecture. Applications to Brownian motion appear in Borell \cite{borell}. 

Some of the methods employed to solve the isoperimetric problem in Gauss space imply isoperimetric comparisons for certain perturbations of the Gaussian measure. For example, Bakry and Ledoux~\cite{bl}, see also Bobkov~\cite{bobkov-perturbation}, obtained a L\'evy-Gromov type inequality by showing that, for probability measures on $\rrn$ having a log-concave density with respect to $\ga_c$, i.e., those of the form 
\begin{equation}
\label{eq:miau}
\exp(\delta(p)-c|p|^2),\quad\text{for some concave function } \delta(p),
\end{equation}
the weighted perimeter of a set is greater than or equal to the perimeter of a half-space with the same volume for the measure $\ga_c$ scaled to have unit volume. Indeed, this comparison is still valid for weighted Riemannian manifolds where the associated Bakry-\'Emery-Ricci curvature has a positive lower bound, see also Bayle~\cite[Ch.~3]{bayle-thesis} and Morgan~\cite{morgandensity}. A recent result of Milman~\cite{milman} contains more general isoperimetric inequalities for weighted Riemannian manifolds depending on a curvature-dimension-diameter condition. Other modifications of Gaussian weights were considered by Fusco, Maggi and Pratelli~\cite{fmp}, who classified the isoperimetric sets for the measures $\exp(-|x|^2/2)\,dx\,dy$ in $\rr^n\times\rr^k$.

\subsection{The partitioning problem for Euclidean densities}
\noindent

In the present paper we study the isoperimetric question in Euclidean open sets for some log-concave perturbations of the Gaussian measure. In order to motivate and describe our results in detail we need to introduce some notation and definitions.

Let $\Om$ be an open subset of $\rrn$. By a \emph{density} on $\Om$ we mean a continuous positive function $f=e^\psi$ on $\Om$ which is used to weight the Hausdorff measures in $\rrn$. In particular, for any open set $E\sub\Om$ with smooth boundary $\ptl E\cap\Om$, the \emph{weighted volume} and the \emph{weighted perimeter in} $\Om$ of $E$ are given by
\begin{align*}
V_f(E):=\int_E f\,dv, \qquad P_f(E,\Om):=\int_{\ptl E\cap\Om} f\,da,
\end{align*}
where $dv$ and $da$ denote the Euclidean elements of volume and area, respectively. Observe that only the area of $\ptl E$ inside $\Om$ is taken into account, so that $\ptl E\cap\ptl\Om$ \emph{has no contribution to} $P_f(E,\Om)$. By following a relaxation procedure as in Ambrosio~\cite{ambrosio-bv} and Miranda~\cite{miranda-bv}, we can define the weighted perimeter $P_f(E,\Om)$ of any Borel set $E\subeq\Om$, see \eqref{eq:wp2}. The resulting perimeter functional satisfies some of the basic properties of Euclidean perimeter, see  Section~\ref{subsec:perimeter}. The reader interested in the theory of finite perimeter sets and functions of bounded variation in Euclidean domains with density or in more general metric measure spaces is referred to the papers of Bellettini, Bouchitt\'e and Fragal\`a \cite{bellettini}, Baldi \cite{baldi}, Ambrosio~\cite{ambrosio-bv}, and Miranda~\cite{miranda-bv}.

Once we have suitable notions of volume and perimeter, we can investigate the \emph{isoperimetric} or \emph{partitioning problem} in $\Om$, which seeks sets $E\sub\Om$ with the least possible weighted perimeter in $\Om$ among those enclosing the same weighted volume. When such a set exists then it is called a \emph{weighted isoperimetric region} or a \emph{weighted minimizer} in $\Om$. 

The study of isoperimetric problems in Euclidean open sets with density has increased in the last years. However, in spite of the last advances, the characterization of the solutions has been achieved only for some densities having a special form or a nice behaviour with respect to a certain subgroup of diffeomorphisms. In particular, radial and homogeneous densities are being a focus of attention, see the related works \cite{lcdc}, \cite{howe}, \cite{morgan-pratelli}, \cite{figalli}, \cite{chambers}, \cite{cabre3}, \cite{homostable}, \cite{milman-rotem}, and the references therein.

For the Gaussian measure $\ga_c$ the partitioning problem has been considered inside domains with simple geometric properties. On the one hand, Lee \cite{gauss-halfspaces} employed approximation as in the first proofs of the Gaussian isoperimetric inequality to state that, in a half-space $\Om\sub\rrn$, the intersections with $\Om$ of half-spaces perpendicular to $\ptl\Om$ are weighted minimizers for any volume. As we pointed out before, this approach does not provide uniqueness of weighted minimizers. Later on, Adams, Corwin, Davis, Lee and Visocchi~\cite{gauss-sectors} derived some properties of weighted isoperimetric sets inside planar sectors with vertex at the origin, and showed that the half-spaces perpendicular to $\ptl\Om$ are the unique weighted minimizers for a half-space $\Om\sub\rrn$ \emph{with linear boundary}. Since the Gaussian density is invariant under linear isometries of $\rrn$ the proof follows by reflection across $\ptl\Om$ and the characterization of equality cases in the Gaussian isoperimetric inequality given in~\cite{ck}. Unfortunately, this argument does not hold if $0\notin\ptl\Om$, and so the uniqueness of weighted minimizers for an arbitrary half-space $\Om$ remained open. 

On the other hand, Lee~\cite{gauss-halfspaces} studied the partitioning problem when $\Om$ is a horizontal strip symmetric with respect to the $x$-axis in the Gauss plane. In particular, some partial results led to the conjecture \cite[Conj.~5]{gauss-halfspaces}: the weighted minimizers in $\Om$ must be bounded by curves different from line segments and meeting the two components of $\ptl\Om$. However, we must point out that the perimeter-decreasing rearrangement used in \cite[Prop.~5.3]{gauss-halfspaces} to discard isoperimetric curves meeting a vertical segment two or more times leaves invariant a region in $\Om$ bounded by such segments. In particular, this kind of region cannot be excluded as a weighted minimizer in $\Om$. Indeed, inside Gaussian slabs of any dimension, \emph{half-spaces perpendicular to the boundary always provide weighted isoperimetric regions}. As E.~Milman and F.~Barthe explained to us, this can be deduced from the Gaussian isoperimetric inequality by a variety of ways (we give the details later in a more general context). To the best of our knowledge, the uniqueness of weighted minimizers in this setting was still an open question. 

The partitioning problem in a Euclidean half-space endowed with a modification of the Gaussian density was considered by Brock, Chiacchio and Mercaldo in \cite{bcm}. These authors combine an optimal transport argument with the Gaussian isoperimetric inequality to establish that, in the coordinate half-space $\Om:=\rr^n\times\rr^+$ with product density $f(z,t):=t^m\exp(-(|z|^2+t^2)/2)$, where $m\geq 0$, the intersections with $\Om$ of half-spaces perpendicular to $\ptl\Om$ uniquely minimize the weighted perimeter for any weighted volume. Note that the density $f$ can be expressed as $\exp(m\log t-|p|^2/2)$, where $p=(z,t)$. In particular, $f$ is a log-concave perturbation of the Gaussian measure as in \eqref{eq:miau}, with perturbation term depending only on the coordinate function $t$. 

\subsection{The setting. Results and proofs for half-spaces and slabs}
\noindent

Motivated by the previous works, in this paper we investigate the partitioning problem in a more general setting that we now describe. For an open set $\Om\subeq\rrn$, we consider one-dimensional perturbations of the Gaussian density $\ga_c$ depending only on the signed distance from a fixed linear hyperplane. Since $\ga_c$ is invariant under linear isometries of $\rrn$ there is no loss of generality in assuming that the hyperplane is the horizontal one of equation $x_{n+1}=0$. Hence we are interested in product densities of the form
\begin{equation}
\label{eq:mefisto}
f(p):=\exp\big(\omega(\pi(p))-c|p|^2\big)=\exp(-c|z|^2)\,\exp(\omega(t)-ct^2), \quad p=(z,t),
\end{equation}
where $\pi:\rrn\to\rr$ is the projection onto the vertical axis, $\omega$ is a continuous function on $J_\Om:=\pi(\Om)$, and $c$ is a positive constant. Equation~\ref{eq:mefisto} defines a very large family of Euclidean densities which includes $\ga_c$ and the ones studied in the work of Brock, Chiacchio and Mercaldo~\cite{bcm}.

In Section~\ref{subsec:pgd} we obtain some basic properties for a density $f$ as in \eqref{eq:mefisto}. A straightforward computation shows that the perturbation term $\omega(\pi(p))$ is concave as a function of $p\in\Om$ if and only if $\omega(t)$ is a concave function of $t\in J_\Om$. Thus, the concavity of $\omega$ implies that $f$ is a log-concave perturbation as in \eqref{eq:miau} of the Gaussian measure $\ga_c$. Moreover, as we see in Lemma~\ref{lem:finitevolume}, this entails finite weighted volume for $\Om$ and finite weighted perimeter for any half-space intersected with $\Om$. Hence, from a well-known result that we recall in Theorem~\ref{th:exist}, we deduce existence of weighted isoperimetric regions in $\Om$ for any weighted volume. By these reasons (and some more that we will explain later), we are naturally lead to assume concavity of $\omega$ for densities as in \eqref{eq:mefisto}.

In the previous situation, the first important result is the following: inside an open half-space or slab $\Om:=\rr^n\times (a,b)$ with density $f$ as in \eqref{eq:mefisto} and $\omega$ concave, \emph{half-spaces perpendicular to $\ptl\Om$ always provide weighted minimizers}. As we recently heard from E.~Milman and F.~Barthe, when $f$ is a product of probability densities the associated isoperimetric profile $I_{\Om,f}$ defined in \eqref{eq:isopro} coincides with the Gaussian isoperimetric profile $I_{\ga}$. The comparison $I_{\Om,f}\geq I_\ga$ is a consequence of the L\'evy-Gromov inequality previously mentioned. Moreover, it can be deduced also from tensorization properties as in Barthe and Maurey~\cite[Remark after Prop.~5]{bm}, symmetrization with respect to a model measure as in Ros~\cite[Thms.~22 and 23]{rosisoperimetric}, or optimal transport methods as in Milman~\cite[Thm.~2.2]{milman-spectral}. On the other hand, by inspecting a half-space $H$ perpendicular to $\ptl\Om$ and using the Gaussian isoperimetric inequality, we conclude that $I_{\Om,f}=I_{\ga}$ and that $H$ is a weighted minimizer.

Our main aim in this work is to employ variational arguments, which are independent on the Gaussian isoperimetric inequality, to provide the characterization of second order minima of the weighted perimeter for fixed weighted volume, and the uniqueness of weighted isoperimetric regions when $\omega$ is concave and smooth. More precisely, one of our main results is the following (Theorem~\ref{th:isoperimetric}):
\begin{quotation}
\emph{Let $\Om:=\rr^n\times (a,b)$ be an open half-space or slab with density as in \eqref{eq:mefisto}, where $\omega$ is smooth and concave on the closure of $(a,b)$. Then, the intersections with $\Om$ of half-spaces perpendicular to $\ptl\Om$ uniquely minimize the weighted perimeter in $\Om$ for fixed weighted volume.}
\end{quotation}
In the particular case of the Gaussian measure $\ga_c$, our theorem gives uniqueness of weighted minimizers in arbitrary half-spaces and slabs, thus improving the aforementioned results in \cite{gauss-halfspaces} and \cite{gauss-sectors}. 

Unlike previous approaches, our methods do not heavily rely on the Gaussian isoperimetric inequality (indeed, it follows from our techniques, as we explain later in this Introduction). The proof of Theorem~\ref{th:isoperimetric} goes as follows. The smoothness hypothesis on $\omega$  allows us to invoke some known results from Geometric Measure Theory to obtain in Section~\ref{sec:isoperimetric} existence of weighted isoperimetric regions with nice interior boundaries. In Section~\ref{sec:stable} we use variational tools to study \emph{stable sets}, i.e., second order minima of the weighted perimeter for deformations preserving the weighted volume. We are able to classify these sets under some regularity conditions, by showing that they are all half-spaces parallel or perpendicular to $\ptl\Om$. Finally, in Section~\ref{sec:main} we utilize integration of second order differential inequalities to show that perpendicular half-spaces are isoperimetrically better than the parallel ones. 

Now, we shall explain in more detail the different ingredients of the proof and other interesting results for general densities on Euclidean open sets.

Let $\Om$ be an open subset of $\rrn$ with density $f=e^\psi$. In general, the existence of weighted minimizers in $\Om$ is a non-trivial question, see for instance \cite{rcbm}, \cite{morgan-pratelli} and \cite{cvm}. However, if the weighted volume of $\Om$ is finite, then we can ensure that there are weighted isoperimetric regions in $\Om$ of any weighted volume (Theorem~\ref{th:exist}). Moreover, by the regularity results in \cite{go-ma-ta}, \cite{gruter} and \cite{morgan-reg}, if $\Om$ has smooth boundary and $f\in C^\infty(\barOm)$ then, for any weighted minimizer $E$, the interior boundary $\overline{\ptl E\cap\Om}$ is a disjoint union $\Sg\cup\Sg_0$, where $\Sg$ is a smooth embedded hypersurface, and $\Sg_0$ is a closed set of singularities with Hausdorff dimension less than or equal to $n-7$ (Theorem~\ref{th:reg}). As we pointed out before, for a density $f$ as in \eqref{eq:mefisto} with $\omega$ smooth and concave, the weighted volume of any smooth open set $\Om$ is finite (Lemma~\ref{lem:finitevolume}), and so weighted minimizers with sufficiently regular interior boundaries exist in $\Om$. 

On the other hand, the Gaussian measure shows that the interior boundary of a minimizer need not be bounded, see some related results in \cite{rcbm}, \cite{morgan-pratelli} and \cite{pratelli-cinti}. This lack of compactness leads us to study the more general condition of \emph{null weighted capacity}. The capacity of a compact set in a Riemannian manifold was introduced by Choquet, and it physically represents the total electrical charge that the set can hold while it maintains a certain potential energy. The reader interested in the theory of capacities and its wide range of applications in isoperimetric problems, harmonic analysis and Brownian motion is referred to the surveys of Grigor'yan~\cite{grigorian} and Troyanov~\cite{troyanov}. In a very recent work, Grigor'yan and Masamune~\cite{gri-masa} have extended the notion of capacity and obtained parabolicity results for weighted Riemannian manifolds.

In our setting, a hypersurface $\Sg$ has \emph{null weighted capacity} if the quantity $\text{Cap}_f(K)$ defined in \eqref{eq:capacity} vanishes for any compact set $K\subeq\Sg$. It is clear that a compact hypersurface $\Sg$ satisfies $\text{Cap}_f(\Sg)=0$, but the reverse statement is not true. This is illustrated in Example~\ref{ex:finite}, where it is shown that $\text{Cap}_f(\Sg)=0$ for any complete hypersurface $\Sg$ of finite weighted area. Hypersurfaces of null weighted capacity play a crucial role in our classification of stable sets in Section~\ref{sec:stable} due to an integral \emph{stability inequality} valid for functions that need not vanish in $\Sg$, see Proposition~\ref{prop:extended}. As we aim at studying the isoperimetric problem by means of the stability condition, it is then important to establish a relation between weighted minimizers and hypersurfaces of null weighted capacity. This is done in Theorem~\ref{th:nullcap}, where we prove that, for an arbitrary smooth bounded density $f$ on a smooth open set $\Om\subeq\rrn$, the regular part $\Sg$ of the interior boundary $\overline{\ptl E\cap\Om}$ of a weighted minimizer $E$ is a hypersurface of null weighted capacity. The main difficulty in proving this theorem is the possible presence in high dimensions of a \emph{noncompact} singular set $\Sg_0$. However, we have been able to adapt the arguments given for constant densities and compact interior boundaries by Sternberg and Zumbrun~\cite{sz}. This requires uniform estimates for the perimeter of a weighted minimizer inside Euclidean balls (Proposition~\ref{prop:growth}) and the fact that the $(n-2)$-dimensional Hausdorff measure of $\Sg_0$ vanishes.

We now turn to describe our stability results in Euclidean domains with density. Let $\Om\subeq\rrn$ be a smooth open set with a density $f=e^\psi$ smooth on $\barOm$. Motivated by the regularity properties of a weighted minimizer in Theorem~\ref{th:reg} and the null weighted capacity property of interior boundaries in Theorem~\ref{th:nullcap}, we consider open sets $E\sub\Om$ with finite weighted perimeter and such that $\overline{\ptl E\cap\Om}=\Sg\cup\Sg_0$, where $\Sg$ is a smooth embedded hypersurface with $\text{Cap}_f(\Sg)=0$, and $\Sg_0$ is a closed singular set with vanishing weighted area. If $\Sg$ has non-empty boundary then we assume $\ptl\Sg=\Sg\cap\ptl\Om$. We say that $E$ is \emph{weighted stable} if it is a critical point with non-negative second derivative for the weighted perimeter functional under compactly supported variations moving $\ptl\Om$ along $\ptl\Om$ and preserving the weighted volume. Thus, if $E$ is weighted stable, then $\Sg$ is an \emph{$f$-stable free boundary hypersurface} as defined by Castro and the author in \cite{castro-rosales} (for constant densities this is just the classical notion of stable constant mean curvature hypersurface with free boundary). As a consequence, we can apply the formulas in \cite{castro-rosales} for the first and second derivatives of weighted area and volume in order to deduce some variational properties of a weighted stable set. 

On the one hand, the hypersurface $\Sg$ has constant \emph{$f$-mean curvature} and meets $\ptl\Om$ orthogonally along $\ptl\Sg$. The \emph{$f$-mean curvature} of $\Sg$ is the function $H_f$ in \eqref{eq:fmc} previously introduced by Gromov~\cite{gromov-GAFA} when computing the first derivative of the weighted area. On the other hand, the associated \emph{$f$-index form} of $\Sg$ defined in \eqref{eq:index1} is nonnegative for smooth functions having compact support on $\Sg$ and mean zero with respect to the weighted element of area. Note that the $f$-index form involves the extrinsic geometry of $\Sg$, the second fundamental form of $\ptl\Om$ and the \emph{Bakry-\'Emery-Ricci curvature} $\ric_f$ in \eqref{eq:fricci}. The $2$-tensor $\ric_f$ was first introduced by Lichnerowicz \cite{lich1}, \cite{lich2}, and later generalized by Bakry and \'Emery \cite{be} in the framework of diffusion generators. In particular, it is easy to observe that the stability inequality becomes more restrictive provided the ambient set $\Om$ is convex and $\ric_f\geq 0$. By assuming both hypotheses we deduce in Proposition~\ref{prop:extended} \emph{more general stability inequalities} for mean zero functions satisfying certain integrability conditions. This allows us to deform a given stable set by means of infinitesimal variations \emph{that could move the singular set $\Sg_0$}. The proof of Proposition~\ref{prop:extended} relies on the null weighted capacity property, which permits to extend the approximation arguments employed by Ritor\'e and the author for constant densities in convex solid cones, see  \cite{cones}

Coming back to the case where $\Om=\rr^n\times (a,b)$ and $f$ is a density as in \eqref{eq:mefisto} with $\omega$ smooth on the closure of $(a,b)$, we analyze in Lemma~\ref{lem:half-spaces} when a half-space $E$ intersected with $\Om$ is a weighted stable set. We get that a necessary condition for $E$ to be weighted stable is that $\ptl E$ is either parallel or perpendicular to $\ptl\Om$. Furthermore, the spectral gap inequality in Gauss space yields that half-spaces parallel to $\ptl\Om$ are weighted stable if and only if $\omega$ is convex. We must remark that the weighted stability of half-spaces for Euclidean product measures was also studied by Barthe, Bianchini and Colesanti~\cite{chiara}, and Doan~\cite{calibrations}. 

From the previous discussion we conclude that half-spaces parallel to $\ptl\Om$ can be discarded as weighted minimizers if $\omega$ is strictly concave. Moreover, the concavity of $\omega$ also implies $\ric_f\geq 2c>0$, which allows to apply the generalized stability inequalities in Proposition~\ref{prop:extended}. Motivated by all this, we are led to study in more detail the stability condition when $\omega$ is a concave function. As a consequence of our analysis, in Theorem~\ref{th:stable} we establish the following classification, which is one of our main results:

\begin{quotation}
\emph{Let $\Om:=\rr^n\times (a,b)$ be an open half-space or slab with density as in \eqref{eq:mefisto}, where $\omega$ is smooth and concave on the closure of $(a,b)$. Then, a weighted stable set of finite weighted perimeter and such that the regular part of the interior boundary has null weighted capacity is the intersection with $\Om$ of a half-space parallel or perpendicular to $\ptl\Om$.}
\end{quotation}

The proof of this theorem employs the generalized stability inequalities in Proposition~\ref{prop:extended} with suitable deformations of $E$. More precisely, the fact that $\ptl\Om$ is totally geodesic leads us to consider equidistant sets translated along a fixed direction $\eta$ of $\ptl\Om$ to keep the weighted volume constant. From an analytical point of view we show in Lemma~\ref{lem:test} that these deformations have associated test functions of the form $u:=\alpha+h$, where $\alpha$ is a real constant and $h$ is the normal component of $\eta$. After evaluating the $f$-index form over $u$, it turns out that this kind of variation decreases the weighted perimeter unless the interior boundary of $E$ is a hyperplane intersected with $\Om$. From here, it is not difficult to conclude that this hyperplane must be parallel or perpendicular to $\ptl\Om$, which completes the proof.  We must remark that similar test functions already appeared in previous stability results, see for example Sternberg and Zumbrun~\cite{sz2}, and Barbosa, do Carmo and Eschenburg~\cite{bdce}. More recently, McGonagle and Ross~\cite{stable-Gauss} have used the same functions to describe smooth, complete, orientable hypersurfaces of constant $f$-mean curvature and finite index in $\rrn$ with Gaussian density $\ga_c$.

Our stability theorem contains the classification of weighted stable sets in Gaussian half-spaces and slabs allowing the presence of singularities, see Corollary~\ref{cor:stableGausshalf}. Indeed, from the techniques of the proof we can also deduce characterization results for smooth, complete, orientable $f$-stable hypersurfaces with finite weighted area and free boundary in a half-space or slab, see Corollary~\ref{cor:stableGauss}. We must emphasize that the \emph{finite area hypothesis} is very restrictive in $\rrn$ with constant density since it enforces a complete constant mean curvature hypersurface to be compact. However, for general densities this is no longer true, as it is shown for example by a Gaussian hyperplane. More recent results for free boundary $f$-stable hypersurfaces are found in \cite{homostable} and \cite{castro-rosales}. On the other hand, complete $f$-minimal hypersurfaces with non-negative second derivative of the weighted area for any variation have been intensively studied, see the introduction of \cite{castro-rosales} for a very complete list of references. 

Once we have ensured existence of isoperimetric minimizers and described the stable candidates, we only have to compare their weighted perimeter for fixed weighted volume. This is done in Proposition~\ref{prop:comparison}, where we finally show that half-spaces perpendicular to $\ptl\Om$ are better than half-spaces parallel to $\ptl\Om$. The key ingredient of the proof is a second order differential inequality, which is satisfied by the relative profile associated to the family of half-spaces parallel to $\ptl\Om$. Since this inequality becomes an equality for half-spaces perpendicular to $\ptl\Om$, the desired comparison follows from a classical integration argument. 
As any two half-spaces perpendicular to $\ptl\Om$ of fixed weighted volume have the same weighted perimeter, our uniqueness result for weighted minimizers is the best that can be expected in this setting. We must also mention that the integration of differential inequalities was already employed in several isoperimetric comparisons, see for instance Morgan and Johnson~\cite{morgan-johnson}, Bayle~\cite[Ch.~3]{bayle-thesis}, \cite{bayle-paper}, and Bayle and the author~\cite{bayle-rosales}.

\subsection{Results for $\Om=\rrn$}
\noindent

The methods employed to prove Theorems~\ref{th:isoperimetric} and \ref{th:stable} provide also characterization results for weighted stable and isoperimetric sets in $\rrn$ with a density $f$ as in \eqref{eq:mefisto}. More precisely, in Theorem~\ref{th:stableGausswhole} we establish that, if $\omega$ is smooth and concave, then any weighted stable set of finite weighted perimeter, null weighted capacity and small singular set is a half-space. Moreover, if $\omega$ is not an affine function, then the half-space must be horizontal or vertical. From this result we deduce in Theorem~\ref{th:isoperimetric2} that, if $\omega$ is concave but not affine, then vertical half-spaces are the unique weighted minimizers. On the other hand, when $\omega$ is affine, it follows that (arbitrary) half-spaces uniquely minimize the weighted perimeter for fixed weighted volume. 

In the particular case of the Gaussian density $\ga_c$ on $\rrn$ our stability theorem implies that any weighted stable set $E$ in the conditions defined above must be a half-space. When the set $E$ has smooth boundary this also follows from a recent result of McGonagle and Ross~\cite{stable-Gauss} showing that the hyperplanes are the unique smooth, complete and orientable $f$-stable hypersurfaces in the Gauss space. Unfortunately, the solution to the isoperimetric problem for $\ga_c$ cannot be deduced from their theorem, since weighted minimizers need not be smooth in high dimensions. However, our stability result allows the presence of a small singular set in such a way that it can be applied to any weighted isoperimetric set. As a consequence our techniques provide a new proof, of possible independent interest, of the Gaussian isoperimetric inequality, leading at once to the characterization of equality cases. 

Recently F.~Morgan observed that the techniques in this paper may be employed to produce stability and isoperimetric results in $\rr^2$ with a Riemannian metric depending only on some of the two coordinates (maybe decreasing or concave), and endowed with unit or Gaussian density. For the moment, the author has no progress about this problem. 

\subsection{The non-smooth case}
\noindent

At this point, it is worth to recall that our main results in Theorems~\ref{th:isoperimetric} and \ref{th:stable} assume smoothness and concavity in the closure of $(a,b)$ of the perturbation term $\omega$ in \eqref{eq:mefisto}. Indeed, in order to apply the variational approach followed in Section~\ref{sec:stable} to characterize weighted stable sets, we only need $\Sg$ to be a $C^3$ hypersurface. So, by the regularity results in \cite{morgan-reg}, it suffices to suppose $\omega\in C^{2,\alpha}$ in order to prove Theorem~\ref{th:isoperimetric}.

In Section~\ref{subsec:non-smooth} we discuss the isoperimetric question in a half-space or slab $\Om:=\rr^n\times (a,b)$ endowed with a density $f$ as in \eqref{eq:mefisto}, where $\omega$ is a concave function, possibly non-smooth on the closure of $(a,b)$. As we have already mentioned, a previous work in this direction is due to Brock, Chiacchio and Mercaldo~\cite{bcm}, who established that, in the half-space $\Om:=\rr^n\times\rr^+$ with density $f(z,t):=t^m\,\exp(-(|z|^2+t^2)/2)$, where $m\geq 0$, the half-spaces perpendicular to $\ptl\Om$ uniquely minimize the weighted perimeter for fixed weighted volume. For the proof, the authors employed a $1$-Lispchitz map $T:\rrn\to\Om$ equaling the identity on $\rr^n\times\{0\}$ and pushing the Gaussian measure forward the weighted volume associated to $f$. From here, the isoperimetric comparison in $\Om$ follows from the Gaussian isoperimetric inequality. We must remark that this result in \cite{bcm} is not a consequence of our Theorem~\ref{th:isoperimetric} since the perturbation term $\omega(t):=m\log(t)$ is not smooth on $\rr^+_0$ when $m>0$. After circulating this manuscript we heard from E.~Milman that the previous approach also leads to the isoperimetric property of half-spaces perpendicular to $\ptl\Om$ for any concave function $\omega$. Moreover, by reproducing the arguments in \cite[Thm.~2.1]{bcm} and \cite[Thm.~2.2]{milman-spectral} we prove in Theorem~\ref{th:nuevo} that any weighted minimizer must be a half-space parallel or perpendicular to $\ptl\Om$. In addition, when $\omega$ is smooth on $(a,b)$ we can invoke our comparison result in Proposition~\ref{prop:comparison} to deduce that half-spaces perpendicular to $\ptl\Om$ are the unique weighted isoperimetric regions. This clearly extends the isoperimetric result in \cite{bcm}. Finally, Theorem~\ref{th:nuevo2} treats the case $\Om=\rrn$.

\subsection{Organization}
\noindent

The paper is organized as follows. In Section~\ref{sec:preliminaries} we introduce some preliminary material about Euclidean densities, hypersurfaces of null weighted capacity and weighted perimeter. In Section~\ref{sec:isoperimetric} we recall existence and regularity results for weighted minimizers, and show that the regular part of the boundary of a weighted minimizer is a hypersurface of null weighted capacity. Section~\ref{sec:stable} is devoted to stability inequalities and characterization results for weighted stable sets. Finally, Section~\ref{sec:main} contains our classification of weighted minimizers.  

\subsection*{Acknowledgements}
I would like to thank F.~Morgan for some comments and suggestions. I also thank E.~Milman and F.~Barthe, whose valuable comments about optimal transport and product measures helped me improve this work.

\section{Preliminaries}
\label{sec:preliminaries}
\setcounter{equation}{0}

In this section we introduce the notation and list some basic results that will be used throughout the paper.

\subsection{Euclidean densities}
\label{subsec:mwd}
\noindent

We consider Euclidean space $\rrn$ endowed with its standard Riemannian flat metric $\escpr{\cdot\,,\cdot}$. Given an open set $\Om\subeq\rrn$, by a \emph{density} on $\Om$ we mean a continuous positive function $f=e^\psi$ defined on $\Om$. This function is used to weight the Euclidean Hausdorff measures. In particular, for a Borel set $E\subeq\Om$, the \emph{weighted volume} and \emph{weighted area} of $E$ are respectively defined by 
\begin{align}
\label{eq:volume}
V_f(E)&:=\int_E dv_f=\int_E f\,dv,
\\
\label{eq:area}
A_f(E)&:=\int_{E}da_f=\int_{E}f\,da,
\end{align}
where $dv$ and $da$ are the Lebesgue measure and the $n$-dimensional Hausdorff measure in $\rrn$. We shall employ the notation $dl_f:=f\,dl$ for the $(n-1)$-dimensional weighted Hausdorff measure. We will say that $f$ has \emph{finite weighted volume} if $V_f(\Om)<+\infty$.

For a smooth density $f=e^\psi$ on $\barOm$, the associated \emph{Bakry-\'Emery-Ricci tensor} is the $2$-tensor
\begin{equation}
\label{eq:fricci}
\text{Ric}_f:=-\nabla^2\psi,
\end{equation}
where $\nabla^2$ is the Euclidean Hessian operator. Clearly, if the density is constant, then $\ric_f=0$. The \emph{Bakry-\'Emery-Ricci curvature} at a point $p\in\barOm$ in the direction of a unit vector $w\in\rrn$ is given by $(\ric_f)_p(w,w)$. If this curvature is always greater than or equal to a constant $2c$, then we write $\ric_f\geq 2c$.

\subsection{Perturbations of the Gaussian density}
\label{subsec:pgd}
\noindent

For any $c>0$, the associated \emph{Gaussian density} in $\rrn$ is defined by $\ga_c(p):=e^{-c |p|^2}$. Observe that $\ga_c$ is bounded and radial, i.e., it is constant over round spheres centered at the origin. Hence, $\ga_c$ is invariant under Euclidean linear isometries, which implies by \eqref{eq:volume}, \eqref{eq:area}, and the change of variables theorem, that the weighted volume and area of $E$ and $\phi(E)$ coincide for any Borel set $E$ and any linear isometry $\phi$. It is also known that $\ga_c$ has finite weighted volume, and that any hyperplane of $\rrn$ has finite weighted area. Moreover, by using \eqref{eq:fricci} we get that $\ga_c$ has constant Bakry-\'Emery-Ricci curvature $2c$. 

In this paper we are interested in certain Euclidean measures having a log-concave density with respect to $\ga_c$. Let $\pi:\rrn\to\rr$ be the projection onto the vertical axis. Given an open set $\Om\subeq\rrn$ we denote $J_\Om:=\pi(\Om)$. For any continuous function $\omega:J_\Om\to\rr$, and any $c>0$, we consider the following density on $\Om$
\begin{equation}
\label{eq:perturb}
f(p)=e^{\psi(p)}, \ \text{ where } \ \psi(p):=\omega(\pi(p))-c |p|^2.
\end{equation}
Clearly $\omega(\pi(p))$ is constant over horizontal hyperplanes. It is also clear that $f$ coincides with the restriction of $\ga_c$ to $\Om$ when $\omega=0$. If we suppose that $\omega$ is smooth on $\overline{J}_\Om$, and denote by $\omega'$ and $\omega''$ the first and second derivatives of $\omega(t)$ with respect to $t$, then a straightforward computation from \eqref{eq:perturb} and \eqref{eq:fricci} shows that
\begin{equation}
\label{eq:riccperturb}
(\nabla\psi)_p=\omega'(\pi(p))\,\ptl_t-2c p, \qquad 
(\ric_f)_p(w,w)=-\omega''(\pi(p))\,\escpr{\ptl_t,w}^2+2c |w|^2,
\end{equation}
for any point $p\in\overline{\Om}$ and any vector $w\in\rrn$. Note that $\ptl_t:=(\nabla\pi)_p=(0,0,\ldots, 1)$. 

From the second equality in \eqref{eq:riccperturb}, it follows that the perturbation term $\omega(\pi(p))$ in \eqref{eq:perturb} is a concave function on $\Om$ if and only if $\omega(t)$ is a concave function of $t\in J_\Om$.  In such a case the density $f$ has the same form as in  \eqref{eq:miau}, and this allows for instance to apply the L\'evy-Gromov type inequality mentioned in the Introduction. In this sense, the concavity of $\omega$ seems a natural hypothesis when we consider densities as in \eqref{eq:perturb}. Furthermore, the concavity property implies that $f$ is bounded from above by a density of Gaussian type. As an immediate consequence, we get the next result.

\begin{lemma}
\label{lem:finitevolume}
Let $\Om$ be an open set of $\rrn$ endowed with a density $f=e^\psi$ as in \eqref{eq:perturb}. If $\omega$ is concave, then $f$ is bounded and $\Om$ has finite weighted volume. Moreover, the intersection with $\Om$ of any hyperplane in $\rrn$ has finite weighted area.
\end{lemma}

\begin{example}
The previous lemma fails if $\omega$ is not assumed to be concave. For example, for $\omega(t):=2c t^2$ we can construct the planar density $f(x,y)=e^{c(y^2-x^2)}$, for which the weighted area of $\rr^2$ is not finite. Note also that any vertical line have infinite weighted length.
\end{example}

\subsection{Weighted divergence theorems}
\label{subsec:divlap}
\noindent

Let $\Om$ be an open set of $\rrn$ endowed with a density $f=e^\psi$ smooth on $\barOm$. For any smooth vector field $X$ on $\overline{\Om}$, the \emph{$f$-divergence} of $X$ is the function
\begin{equation}
\label{eq:divf}
\divv_f X:=\divv X+\escpr{\nabla\psi,X},
\end{equation}
where $\divv$ and $\nabla$ denote the Euclidean divergence of vector fields and the gradient of smooth functions, respectively. The $f$-divergence $\divv_f$ is the adjoint operator of $-\nabla$ for the $L^2$ norm associated to the weighted volume $dv_f$. By using the Gauss-Green theorem together with equality $\divv_f X\,dv_f=\divv(fX)\,dv$, we obtain
\begin{equation}
\label{eq:divdesi}
\int_E\divv_f X\,dv_f=-\int_{\ptl E\cap\Om}\escpr{X,N}\,da_f,
\end{equation}
for any open set $E\sub\Om$ such that $\ptl E\cap\Om$ is a smooth hypersurface, and any smooth vector field $X$ with compact support in $\Om$. In the previous formula $N$ denotes the inner unit normal along $\ptl E\cap\Om$. 

Let $\Sg$ be a smooth oriented hypersurface in $\overline{\Om}$, possibly with boundary $\ptl\Sg$. For any smooth vector field $X$ along $\Sg$, we define the \emph{f-divergence relative to} $\Sg$ of $X$ by
\[
\divv_{\Sg,f}X:=\divv_\Sg X+\escpr{\nabla\psi,X},
\]
where $\divv_\Sg$ is the Euclidean divergence relative to $\Sg$. If $N$ is a unit normal vector along $\Sg$, then the \emph{f-mean curvature} of $\Sg$ with respect to $N$ is the function
\begin{equation}
\label{eq:fmc}
H_f:=-\divv_{\Sg,f}N=nH-\escpr{\nabla\psi,N},
\end{equation}
where $H:=(-1/n)\divv_\Sg N$ is the Euclidean mean curvature of $\Sg$. By using the Riemannian divergence theorem it is not difficult to get
\begin{equation}
\label{eq:divthsup}
\int_\Sg\divv_{\Sg,f}X\,da_f=-\int_\Sg H_f\escpr{X,N}\,da_f-\int_{\ptl\Sg}\escpr{X,\nu}\,dl_f,
\end{equation}
for any smooth vector field $X$ with compact support on $\Sg$, see \cite[Lem.~2.2]{homostable} for details. Here we denote by $\nu$ the \emph{conormal vector}, i.e., the inner unit normal to $\ptl\Sg$ in $\Sg$. From now on, we understand that the integrals over $\ptl\Sg$ are all equal to zero provided $\ptl\Sg=\emptyset$.

Given a function $u\in C^\infty(\Sg)$, the \emph{f-Laplacian relative} to $\Sg$ of $u$ is defined by
\begin{equation}
\label{eq:deltaf}
\Delta_{\Sg,f}u:=\divv_{\Sg,f}(\nabla_\Sg u)=\Delta_\Sg u
+\escpr{\nabla_\Sg\psi,\nabla_\Sg u},
\end{equation} 
where $\nabla_\Sg$ is the gradient relative to $\Sg$. For this operator we have the following integration by parts formula, which is an immediate consequence of  \eqref{eq:divthsup}
\begin{equation}
\label{eq:ibp}
\int_\Sg u_1\,\Delta_{\Sg,f}u_2\,da_f=-\int_\Sg\escpr{\nabla_\Sg u_1,\nabla_\Sg u_2}\,da_f-\int_{\ptl\Sg}u_1\,\frac{\ptl u_2}{\ptl\nu}\,dl_f,
\end{equation}
where $u_1,u_2\in C^\infty_0(\Sg)$ and $\ptl u_2/\ptl\nu$ is the directional derivative of $u_2$ with respect to $\nu$. As usual, we denote by $C_0^\infty(\Sg)$ the set of smooth functions with compact support on $\Sg$. Note that, when $\ptl\Sg\neq\emptyset$, a function in $C^\infty_0(\Sg)$ need not vanish on $\ptl\Sg$.

\subsection{Hypersurfaces of null weighted capacity}
\label{subsec:capacity}
\noindent

Let $\Om$ be an open set of $\rrn$ endowed with a density $f=e^\psi$ smooth on $\barOm$. Given a smooth oriented hypersurface $\Sg\sub\overline{\Om}$, possibly with boundary, and a number $q\geq 1$, we denote by $L^q(\Sg,da_f)$ and $L^q(\ptl\Sg,dl_f)$ the corresponding spaces of integrable functions with respect to the weighted measures $da_f$ and $dl_f$. The \emph{weighted Sobolev space} $H^1(\Sg,da_f)$ is the set of functions $u\in L^2(\Sg,da_f)$ whose distributional gradient $\nabla_\Sg u$ satisfies $|\nabla_\Sg u|\in L^2(\Sg,da_f)$. The \emph{weighted Sobolev norm} of $u\in H^1(\Sg,da_f)$ is $\|u\|:=(\int_\Sg u^2\,da_f+\int_\Sg|\nabla_\Sg u|^2\,da_f)^{1/2}$. We will use the notation $H_0^1(\Sg,da_f)$ for the closure of $C^\infty_0(\Sg)$ with respect to this norm.  

Following Grigor'yan and Masamune \cite{gri-masa}, we define the \emph{weighted capacity} of a compact subset $K\subeq\Sg$ by means of equality
\begin{equation}
\label{eq:capacity}
\text{Cap}_f(K):=\inf\left\{\int_\Sg|\nabla_\Sg u|^2\,da_f\,;\,u\in H^1_0(\Sg,da_f),\,0\leq u\leq 1,\,u=1\text{ in }K\right\}.
\end{equation}
A standard approximation argument shows that we can replace $H^1_0(\Sg,da_f)$ with $C^\infty_0(\Sg)$ in the previous definition. Note also that the monotonicity property $\text{Cap}_f(K_1)\leq\text{Cap}_f(K_2)$ holds for two compact sets $K_1\subeq K_2$. Thus, it is natural to define
\[
\text{Cap}_f(D):=\sup\{\text{Cap}_f(K)\,;\,K\subeq D \text{ is a compact set}\},
\]
for any open set $D\subeq\Sg$. We say that $\Sg$ \emph{has null weighted capacity} if $\text{Cap}_f(\Sg)=0$, i.e., $\text{Cap}_f(K)=0$ for any compact set $K\subeq\Sg$. 
Our main interest in such a hypersurface comes from the next result, which follows from \eqref{eq:capacity} by taking a countable exhaustion of $\Sg$ by precompact open subsets.

\begin{lemma}
\label{lem:capacity}
Let $\Om$ be an open set of $\rrn$ endowed with a density $f=e^\psi$ smooth on $\barOm$. Consider a smooth oriented hypersurface $\Sg\sub\overline{\Om}$, possibly with boundary. If $\emph{Cap}_f(\Sg)=0$, then there is a sequence $\{\var_k\}_{k\in\nn}\sub C^\infty_0(\Sg)$ with:
\begin{itemize}
\item[(i)] $0\leq\var_k\leq 1$ for any $k\in\nn$, 
\item[(ii)] $\lim_{k\to\infty}\var_k(p)=1$ for any $p\in\Sg$, 
\item[(iii)] $\lim_{k\to\infty}\int_\Sg|\nabla_\Sg\var_k|^2\,da_f=0$.
\end{itemize}
\end{lemma}

From the previous lemma we can generalize to hypersurfaces of null weighted capacity some properties and results valid for compact hypersurfaces. In the next result we extend the divergence theorem in \eqref{eq:divthsup} and the integration by parts formula in \eqref{eq:ibp} to vector fields and functions satisfying certain integrability conditions. The proof is obtained from Lemma~\ref{lem:capacity} by a simple approximation argument, see \cite[Lem.~4.4]{cones} for details. 

\begin{lemma}[Generalized divergence theorem and integration by parts formula]
\label{lem:ibpext}
Let $\Om$ be an open set of $\rrn$ endowed with a density $f=e^\psi$ smooth on $\barOm$. Consider a smooth oriented hypersurface $\Sg$, possibly with boundary, and such that $\emph{Cap}_f(\Sg)=0$. Then, for any smooth vector field $X$ on $\Sg$ satisfying
\begin{itemize}
\item[(i)] $|X|\in L^2(\Sg,da_f)$,
\item[(ii)] $\divv_{\Sg,f}X\in L^1(\Sg,da_f)$,
\item[(iii)] $H_f\escpr{X,N}\in L^1(\Sg,da_f)$,
\item[(iv)] $\escpr{X,\nu}\in L^1(\ptl\Sg,dl_f)$,
\end{itemize}
we have equality
\[
\int_\Sg\divv_{\Sg,f}X\,da_f=-\int_\Sg H_f\escpr{X,N}\,da_f-\int_{\ptl\Sg}\escpr{X,\nu}\,dl_f.
\]
As a consequence, for any two functions $u_1,u_2\in C^\infty(\Sg)$ such that
\begin{itemize}
\item[(i)] $u_1$ is bounded,
\item[(ii)] $|\nabla_\Sg u_i|\in L^2(\Sg,da_f)$, $i=1,2$,
\item[(iii)] $\Delta_{\Sg,f}u_2\in L^1(\Sg,da_f)$,
\item[(iv)] $\ptl u_2/\ptl\nu\in L^1(\ptl\Sg,dl_f)$,
\end{itemize}
we have the integration by parts formula
\[
\int_\Sg u_1\,\Delta_{\Sg,f}\,u_2\,da_f=-\int_\Sg\escpr{\nabla_\Sg u_1,\nabla_\Sg u_2}\,da_f-\int_{\ptl\Sg}u_1\,\frac{\ptl u_2}{\ptl\nu}\,dl_f.
\]
\end{lemma}

For a complete hypersurface, the null capacity property can be deduced from a suitable behaviour of the volume growth associated to metric balls centered at a fixed point, see Grigor'yan \cite{grigorian}. A very particular case of this situation is shown in the next example. 

\begin{example}[Complete hypersurfaces of finite area have null capacity]
\label{ex:finite}
Let $\Om$ be an open set of $\rrn$ endowed with a density $f=e^\psi$ smooth on $\barOm$. Consider a smooth, complete, oriented hypersurface $\Sg\sub\overline{\Om}$, possibly with boundary. Fix a point $p_0\in\Sg$ and denote by $D(p_0,r)$ the closed $r$-neighborhood of $p_0$ in $\Sg$ with respect to the intrinsic distance. The completeness of $\Sg$ implies the existence of a sequence of Lipschitz functions $\var_k:\Sg\to [0,1]$ such that $\var_k=1$ in $D(p_0,k/2)$, $\var_k=0$ in $\Sg-D(p_0,k)$ and $|\nabla_\Sg\var_k|\leq\alpha/k$ on $\Sg$ for some positive constant $\alpha$ which does not depend on $k$. Moreover, if $A_f(\Sg)<+\infty$, then we have
\[
\int_\Sg|\nabla_\Sg\var_k|^2\,da_f\leq\alpha^2 \ \frac{A_f(\Sg\cap D(p_0,k))}{k^2}\leq\alpha^2\,\,\frac{A_f(\Sg)}{k^2},
\]
which tends to zero when $k\to\infty$. Finally, given a compact set $K\subeq\Sg$, there is $k_0\in\nn$ such that $\var_k=1$ in $K$ for any $k \geq k_0$. Hence $\text{Cap}_f(K)=0$ by \eqref{eq:capacity}.
\end{example}

\subsection{Sets of finite weighted perimeter}
\label{subsec:perimeter}
\noindent

Let $\Om$ be an open set of $\rrn$ endowed with a density $f=e^\psi$ smooth on $\barOm$. The notion of $f$-divergence in \eqref{eq:divf} allows us to introduce the weighted perimeter of sets by following the classical approach by Caccioppoli and De Giorgi. More precisely, for any Borel set $E\subeq\Om$, the \emph{weighted perimeter of $E$ in $\Om$} is given by
\begin{equation}
\label{eq:wp}
P_f(E,\Om):=\sup\left\{\int_E\divv_f X\,dv_f\,;\,|X|\leq 1\right\},
\end{equation}
where $X$ ranges over smooth vector fields with compact support in $\Om$. Clearly $P_f(E,\Om)$ does not change by sets of volume zero. Thus, we can always assume that $E$ satisfies $0<V_f(E\cap B)<V_f(B)$ for any Euclidean open ball $B$ centered at $\ptl E$, see \cite[Prop.~3.1]{giusti}. 

If $E$ is an open subset of $\Om$ and $\ptl E\cap\Om$ is smooth, then $P_f(E,U)=A_f(\ptl E\cap U)$ for any open set $U\subeq\Om$. This is an immediate consequence of \eqref{eq:wp} and the Gauss-Green formula in \eqref{eq:divdesi}. Indeed, by reasoning as in \cite[Ex.~12.7]{maggi}, we deduce that the previous equality also holds for a set $E$ with almost smooth interior boundary.  In particular, if $\Sg$ is the regular part of $\overline{\ptl E\cap\Om}$ and we assume $\ptl\Sg=\Sg\cap\ptl\Om$, then $P_f(E,\Om)=A_f(\Sg)$ since $A_f(\ptl\Sg)=0$. This shows that $\ptl E\cap\ptl\Om$ does not contribute to $P_f(E,\Om)$. 

We say that $E$ has \emph{finite weighted perimeter in} $\Om$ if $P_f(E,\Om)<+\infty$. We say that $E$ is a \emph{weighted Caccioppoli set in $\Om$} if $P_f(E,U)<+\infty$ for any open set $U\sub\sub\Om$. If $E$ has finite weighted perimeter in $\Om$, then the set function mapping any open set $U\subeq\Om$ to $P_f(E,U)$ is the restriction of a finite Borel measure $m_f$ on $\Om$ (the \emph{weighted perimeter measure}) given by $m_f(Q)=P_f(E,Q):=\inf\{P_f(E,U)\,;\,Q\subeq U\}$, where $Q\subeq\Om$, and $U$ ranges over open subsets of $\Om$. Moreover, by using the Riesz representation theorem as in Thm.~1 of \cite[Sect.~5.1]{evans-gariepy}, we can find a \emph{generalized inner unit normal} $N_f:\Om\to\rrn$ such that the Gauss-Green formula
\begin{equation}
\label{eq:gg2}
\int_E\divv_f X\,dv_f=-\int_{\Om}\escpr{X,N_f}\,dm_f
\end{equation}
holds for any smooth vector field $X$ with compact support in $\Om$. Of course, for an open set with almost smooth interior boundary, the inner unit normal $N$ along the hypersurface $\Sg$ coincides with $N_f$. From \eqref{eq:gg2} it is easy to check that $E$ is a weighted Caccioppoli set in $\Om$ if and only if $E$ is a Caccioppoli set (in Euclidean sense) in $\Om$. 

In general, most of the basic properties of finite perimeter sets in $\rrn$ are extended to sets of finite weighted perimeter. For example, the next lemma can be derived as the localization result stated in \cite[Prop.~3.56]{afp} and \cite[p.~196]{evans-gariepy} for constant densities, with the difference that our estimates involve the weighted perimeter \emph{off of a Euclidean ball}.

\begin{lemma}
\label{lem:peresti}
Let $\Om$ be an open set of $\rrn$ endowed with a density $f=e^\psi$ smooth on $\barOm$. Given a Borel set $E\subeq\Om$ of finite weighted perimeter and a point $p_0\in\rrn$, we have
\[
P_f(E-B(p_0,r),\Om)\leq P_f(E,\Om-B(p_0,r))+A_f(E\cap \ptl B(p_0,r))
\]
for almost every $r>0$. Here $B(p_0,r)$ is the Euclidean open ball of radius $r$ centered at $p_0$.
\end{lemma}

Another result that we need is a regularity property for sets of finite weighted perimeter. The proof is a direct consequence of \cite[Thm.~4.11]{giusti} and the fact that weighted Caccioppoli sets coincide with unweighted Caccioppoli sets.

\begin{lemma}
\label{lem:giusti}
Let $\Om$ be an open set of $\rrn$ endowed with a density $f=e^\psi$ smooth on $\barOm$. Suppose that $E\subeq\Om$ is a Borel set of finite weighted perimeter in $\Om$ such that the generalized inner unit normal $N_f$ of $E$ exists and it is continuous on $\ptl E\cap\Om$. Then $\ptl E\cap\Om$ is a $C^1$ hypersurface.
\end{lemma}

Observe that we have defined in \eqref{eq:wp} the weighted perimeter functional for densities $f=e^\psi$ \emph{which are smooth on} $\barOm$. In the case where $f$ is merely continuous on $\Om$ we can use a relaxation procedure as in the papers of Ambrosio~\cite{ambrosio-bv} and Miranda~\cite{miranda-bv} to define the \emph{weighted perimeter} of a Borel set $E\subeq\Om$ by 
\begin{equation}
\label{eq:wp2}
P_f(E,\Om):=\inf\bigg\{\liminf_{k\to\infty}A_f(\ptl E_k\cap\Om)\,;\,\{E_k\}_{k\in\nn}\to E\,\text{ in }L^1(\Om,dv_f)\bigg\},
\end{equation}
where $\{E_k\}_{k\in\nn}$ ranges over sequences of open sets in $\Om$ such that $\ptl E_k\cap\Om$ is smooth.

In the next result we gather some basic properties of the above perimeter functional that will be used in Section~\ref{subsec:non-smooth} of the paper. For the proof the reader is referred to \cite[Prop.~3.6]{miranda-bv} and \cite[Thm.~3.2]{baldi}.

\begin{lemma}
\label{lem:percont}
Let $\Om$ be an open set of $\rrn$ endowed with a continuous density $f=e^\psi$. Then, the following facts hold:
\begin{itemize}
\item[(i)] $($Lower semicontinuity$)$ If $\{E_k\}_{k\in\nn}$ is a sequence of Borel sets in $\Om$ which converges in $L^1(\Om,dv_f)$ to a Borel set $E$, then $P_f(E,\Om)\leq\liminf_{k\to\infty}P_f(E_k,\Om)$.
\item[(ii)] $($Perimeter of smooth sets$)$ $P_f(E,\Om)=A_f(\ptl E\cap\Om)$, for any open set $E\sub\Om$ such that $\ptl E\cap\Om$ is smooth.
\item[(iii)] $($Approximation by smooth sets$)$ For any Borel set $E\sub\Om$, there is a sequence $\{E_k\}_{k\in\nn}$ of open sets in $\Om$ such that $\ptl E_k\cap\Om$ is smooth, $\{E_k\}_{k\in\nn}\to E$ in $L^1(\Om,dv_f)$, and $\lim_{k\to\infty} P_f(E_k,\Om)=P_f(E,\Om)$.
\end{itemize}
\end{lemma}

We finally remark that standard density arguments imply that, if the function $f=e^\psi$ is smooth on $\barOm$, then the two definitions of weighted perimeter in \eqref{eq:wp} and \eqref{eq:wp2} coincide.

\section{Isoperimetric sets: existence, regularity and null capacity property of the interior boundary}
\label{sec:isoperimetric}

In this section we first recall some known results about the existence and regularity of solutions for the \emph{weighted isoperimetric problem} inside Euclidean open sets with density. Then, we will use technical arguments to show that, for bounded densities, the regular part of the interior boundary of such solutions has null weighted capacity.

Let $\Om$ be an open subset of $\rrn$ endowed with a continuous density $f=e^\psi$. The \emph{weighted isoperimetric profile} of $\Om$ is the function $I_{\Om,f}:(0,V_f(\Om))\to\rr^+_0$ given by
\begin{equation}
\label{eq:isopro}
I_{\Om,f}(v):=\inf\{P_f(E,\Om)\,;\,E\sub\Om\text{ is a Borel set with }V_f(E)=v\},
\end{equation}
where $V_f(E)$ is the weighted volume in \eqref{eq:volume} and $P_f(E,\Om)$ the weighted perimeter in \eqref{eq:wp2}. A \emph{weighted isoperimetric region} or \emph{weighted minimizer} of volume $v\in(0,V_f(\Om))$ is a Borel set $E\sub\Om$ with $V_f(E)=v$ and $P_f(E,\Om)=I_{\Om,f}(v)$.

The existence of weighted minimizers is a non-trivial question. In the works of Bayle, Ca\~nete, Morgan and the author~\cite{rcbm}, and Morgan and Pratelli~\cite{morgan-pratelli}, we can find some elementary examples showing that minimizers need not exist if $\Om$ is unbounded. These papers also provide sufficient conditions for existence involving the growth of the density at infinity, see \cite[Sect.~2]{rcbm} and \cite[Sects.~3 and 7]{morgan-pratelli}. Here we will only use the following result, whose proof relies on the lower semicontinuity of the weighted perimeter and standard compactness arguments, see Morgan \cite[Sects.~5.5. and 9.1]{gmt}, Ca\~nete, Miranda and Vittone~\cite[Prop.~2.2]{cvm}, and Milman \cite[Sect.~2.2]{milman}.

\begin{theorem}[Existence]
\label{th:exist}
Let $\Om$ be an open set of $\rrn$ endowed with a continuous density $f=e^\psi$. If $V_f(\Om)<+\infty$, then there are weighted isoperimetric regions in $\Om$ of any given volume.
\end{theorem}

As pointed out by Morgan \cite[Sect.~3.10]{morgan-reg} the regularity of weighted isoperimetric regions inside a smooth open set $\Om$ with smooth density $f$ is the same as for the classical setting of constant density $f=1$. The latter was studied by Gonzalez, Massari and Tamanini, who obtained interior regularity \cite{go-ma-ta}, and by Gr\"uter, who proved regularity at the free boundary \cite{gruter}. We gather their results in the next theorem, see also \cite[Sect.~2]{milman}.

\begin{theorem}[Regularity]
\label{th:reg}
Let $\Om$ be a smooth open set of $\rrn$ endowed with a density $f=e^\psi$ smooth and positive in $\overline{\Om}$. If $E$ is a weighted isoperimetric region in $\Om$, then the interior boundary $\overline{\ptl E\cap\Om}$ is a disjoint union $\Sg\cup\Sg_0$, where $\Sg$ is a smooth embedded hypersurface with $($possibly empty$)$ boundary $\ptl\Sg=\Sg\cap\ptl\Om$, and $\Sg_0$ is a closed set of singularities with Hausdorff dimension less than or
equal to $n-7$. 
\end{theorem}

\begin{remark}
\label{re:isop}
From the previous theorem, and taking into account that the weighted perimeter in \eqref{eq:wp} does not change by sets of volume zero, we can always assume that a weighted isoperimetric region $E$ in $\Om$ is an open set. Note that $E$ has almost smooth interior boundary, so that $P_f(E,\Om)=A_f(\Sg)$ and, more generally, $P_f(E,U)=A_f(\Sg\cap U)$ for any open set $U\subeq\Om$. The condition $\ptl\Sg=\Sg\cap\ptl\Om$ prevents the existence of interior points of $\Sg$ inside $\ptl\Om$.
\end{remark}

Weighted minimizers and their interior boundaries need not be bounded. This is illustrated by the Gaussian density $\ga_c(p):=e^{-c |p|^2}$, $c>0$, for which any weighted isoperimetric region is, up to a set of volume zero, a Euclidean half-space. Some criteria ensuring boundedness have been found by Bayle, Ca\~nete, Morgan and the author~\cite[Sect.~2]{rcbm}, Morgan and Pratelli~\cite[Sect.~5]{morgan-pratelli}, and Cinti and Pratelli{ \cite{pratelli-cinti}.  In spite of the possible lack of compactness we will be able to show that, for smooth bounded densities, the regular part of the interior boundary is a hypersurface of null weighted capacity. To prove this fact we first establish a uniform upper estimate for the perimeter of a weighted minimizer inside open balls of $\rrn$. 

\begin{proposition}[Upper $n$-Ahlfors-regularity of the interior boundary]
\label{prop:growth}
Let $\Om$ be an open set of $\rrn$ endowed with a bounded density $f=e^\psi$ smooth on $\barOm$. If $E$ is a weighted isoperimetric region in $\Om$, then there exist constants $C>0$ and $R>0$, such that 
\begin{equation}
\label{eq:aub}
P_f(E,\Om\cap B(p,r))\leq Cr^n,
\end{equation}
for any open ball $B(p,r)\sub\rrn$ of radius $r\leq R$.
\end{proposition}

\begin{proof}
Let $\alpha:=\sup\{f(p)\,;\,p\in\Om\}$. Choose an open ball $B(p_0,R_0)$ such that $\overline{B}(p_0,R_0)\sub\Om-\overline{E}$. Define the constant 
\[
R:=\left(\frac{V_f(B(p_0,R_0))}{\alpha\,\rho_n}\right)^{1/(n+1)},
\]
where $\rho_n$ is the Euclidean volume of $B(0,1)$.

Fix a point $p\in\rrn$ and take an open ball $B(p,r)$ of radius $r\leq R$. Note that
\[
V_f(E\cap B(p,r))\leq V_f(\Om\cap B(p,r))\leq\alpha\,\rho_n\,r^{n+1}\leq V_f(B(p_0,R_0)),
\]
by \eqref{eq:volume} and the definition of $R$. Therefore, there is a unique $r'_p\leq R_0$ such that
\begin{equation}
\label{eq:rprimap}
V_f(E\cap B(p,r))=V_f(B(p_0,r'_p)).
\end{equation}

Consider the set $E':=(E-B(p,r))\cup\overline{B}(p_0,r'_p)$. It is clear that $E'\sub\Om$ and $V_f(E')=V_f(E)$. By using that $E$ is a weighted isoperimetric region together with Lemma~\ref{lem:peresti}, we get the following
\begin{align*}
P_f(E,\Om)&\leq P_f(E',\Om)=P_f(E-B(p,r),\Om)+A_f(\ptl B(p_0,r'_p))
\\
&\leq P_f(E,\Om-B(p,r))+A_f(E\cap\ptl B(p,r))+A_f(\ptl B(p_0,r'_p)),
\end{align*}
for almost every $r\leq R$. On the other hand, we have
\[
P_f(E,\Om\cap B(p,r))+P_f(E,\Om-B(p,r))\leq P_f(E,\Om)
\]
since $P_f(E,\cdot)$ is a finite Borel measure in $\Om$. By combining the two previous inequalities and taking into account the definition of weighted area in \eqref{eq:area}, we obtain
\begin{equation}
\label{eq:ineq1}
P_f(E,\Om\cap B(p,r))\leq A_f(E\cap\ptl B(p,r))+A_f(\ptl B(p_0,r'_p))\leq\alpha\,\la_n\,(r^n+(r_p')^n),
\end{equation}
for almost every $r\leq R$, where $\la_n$ is the Euclidean area of the unit sphere in $\rrn$. 

Finally, let $m_0:=\inf\{f(q)\,;\,q\in B(p_0,R_0)\}$. As $r'_p\leq R_0$ then $f(q)\geq m_0$ for any $q\in B(p_0,r'_p)$. From \eqref{eq:volume} and \eqref{eq:rprimap} we deduce
\[
m_0\,\rho_n\,(r'_p)^{n+1}\leq V_f(B(p_0,r'_p))\leq V_f(\Om\cap B(p,r))\leq\alpha\,\rho_n\,r^{n+1}.
\]
The previous inequality yields $r'_p\leq\beta\, r$, for any $r\leq R$, where $\beta$ is a positive constant which does not depend on $p$ and $r$. Plugging this information into \eqref{eq:ineq1}, we conclude that
\[
P_f(E,\Om\cap B(p,r))\leq\alpha\,\la_n\,(1+\beta^n)\,r^n,
\]
for almost every $r\leq R$. This conclusion easily extends for any $r>0$ by the dominated convergence theorem. The proof is completed.
\end{proof}

\begin{remark}
Inequalities similar to \eqref{eq:aub} have been obtained in different settings, see Bayle \cite[Lem.~A.1.2]{bayle-thesis} and Morgan \cite[Prop.~3.2]{morgan-reg} for Riemannian manifolds, and Leonardi and Rigot \cite[Lem.~5.1]{lr} for Carnot groups. More recently, Cinti and Pratelli have deduced \eqref{eq:aub} in $\Om=\rrn$ endowed with a lower semicontinuous density $f$ bounded from above and below, see the proof of \cite[Thm.~5.7]{pratelli-cinti}.
\end{remark}

Now, we are ready to state and prove our main result in this section. 

\begin{theorem}[Null capacity property for weighted minimizers]
\label{th:nullcap}
Let $\Om$ be a smooth open set of $\rrn$ endowed with a bounded density $f=e^\psi$ smooth on $\barOm$. If $E$ is a weighted isoperimetric region in $\Om$, then the regular part $\Sg$ of the interior boundary $\overline{\ptl E\cap\Om}$ is a hypersurface of null weighted capacity. 
\end{theorem}

\begin{proof}
Recall that $\Sg$ has null weighted capacity if $\text{Cap}_f(K)=0$ for any compact subset $K\subeq\Sg$, see \eqref{eq:capacity}. To prove the theorem it is then enough to construct a sequence $\{\var_k\}_{k\in\nn}\sub H^1_0(\Sg,da_f)$ satisfying:
\begin{itemize}
\item[(i)] $0\leq\var_k\leq 1$, for any $k\in\nn$,
\item[(ii)] $\lim_{k\to\infty}\int_\Sg|\nabla_\Sg\var_k|^2\,da_f=0$,
\item[(iii)] for any compact $K\subeq\Sg$ there is $k_0\in\nn$ such that $\var_k=1$ on $K$, for any $k\geq k_0$.
\end{itemize}
For that, we will adapt the arguments employed by Sternberg and Zumbrun~\cite[Lem.~2.4]{sz} in the case of isoperimetric regions inside bounded Euclidean domains with constant density.

Let $\Sg_0$ be the singular set of $\overline{\ptl E\cap\Om}$. We know from Theorem~\ref{th:reg} that $\Sg_0$ is a closed set with $\h^q(\Sg_0)=0$, for any $q>n-7$. Here $\h^q$ denotes the $q$-dimensional Hausdorff measure in $\rrn$.  By Remark~\ref{re:isop} we get $P_f(E,U)=A_f(\Sg\cap U)$ for any open set $U\subeq\Om$. If $\Sg_0=\emptyset$, then $\Sg=\overline{\ptl E\cap\Om}$, which is a complete hypersurface of finite weighted area. In this case $\Sg$ has null weighted capacity by Example~\ref{ex:finite}. So, we can suppose $\Sg_0\neq\emptyset$. This implies, in particular, that $n\geq 7$ and $\mathcal{H}^{n-2}(\Sg_0)=0$.

Let $C>0$ and $R>0$ be the constants in Proposition~\ref{prop:growth}. Thus, we have
\begin{equation}
\label{eq:ce1}
A_f(\Sg\cap B(p,r))=P_f(E,\Om\cap B(p,r))\leq Cr^n,
\end{equation}
for any open ball $B(p,r)\sub\rrn$ with $r\leq R$.

Fix a number $k\in\nn$. Since $\mathcal{H}^{n-2}(\Sg_0)=0$, we can find a sequence $\{B(p_i,r_i)\}_{i\in\mathbb{N}}$ of open balls in $\rrn$ such that $\Sg_0\sub\cup_{i=1}^\infty B(p_i,r_i/2)$, and
\begin{equation}
\label{eq:ce2}
\sum_{i=1}^\infty r_i^{n-2}<\frac{1}{k}.
\end{equation}
Moreover, we may assume $r_i\leq R$, $2r_i\leq 1/k$, and $\Sg_0\cap B(p_i,r_i/2)\neq\emptyset$ for any $i\in\mathbb{N}$. In particular, inequality \eqref{eq:ce1} holds for any $B(p_i,r_i)$. From the fact that $\Sg_0\cap B(p_i,r_i/2)\neq\emptyset$ for any $i\in\mathbb{N}$, it follows that the set $W_k:=\cup_{i=1}^\infty B(p_i,r_i)$ satisfies
\begin{equation}
\label{eq:dist}
W_k\subeq\{p\in\rrn\,;\,\text{dist}(p,\Sg_0)<1/k\}.
\end{equation}

Take a function $\mu\in C^\infty(\rr)$ with $\mu(s)=0$ if $|s|\leq 1/2$, $\mu(s)=1$ if $|s|\geq 1$, and $0<\mu(s)<1$ otherwise. Denote $\alpha:=\max_{s\in\rr}|\mu'(s)|$.

For any $i\in\mathbb{N}$, let $\zeta_i(p):=\mu(d(p,p_i)/r_i)$, where $d(p,p_i):=|p-p_i|$. We get $\zeta_i\in C^\infty(\rrn)$ with $\zeta_i=0$ in $\overline{B}(p_i,r_i/2)$, $\zeta_i=1$ in $\rrn-B(p_i,r_i)$, and $0<\zeta_i<1$ in $B(p_i,r_i)-\overline{B}(p_i,r_i/2)$. Moreover, we have the gradient estimate
\begin{equation}
\label{eq:ce3}
|\nabla\zeta_i|\leq\alpha/{r_i} \ \text{ in } B(p_i,r_i)-\overline{B}(p_i,r_i/2).
\end{equation}

Let $J\sub\mathbb{N}$ be a finite set. We define $\zeta^J_k:\rrn\to [0,1]$ by $\zeta^J_k(p):=\min\{\zeta_i(p)\,;\,i\in J\}$. Note that $\zeta^J_k$ is a piecewise smooth function. Having in mind \eqref{eq:ce3}, \eqref{eq:ce1} and \eqref{eq:ce2}, we obtain
\begin{align*}
\int_\Sg|\nabla_\Sg\zeta_k^J|^2\,da_f&\leq\int_\Sg|\nabla\zeta_k^J|^2\,da_f
\leq\int_\Sg\bigg(\sum_{i\in J}|\nabla\zeta_i|^2\bigg)\,da_f
\\
\nonumber
&=\sum_{i\in J}\int_{\Sg\cap(B(p_i,r_i)-\overline{B}(p_i,r_i/2))}|\nabla\zeta_i|^2\,da_f
\\
\nonumber
&\leq\alpha^2\,\sum_{i\in J}\frac{A_f(\Sg\cap B(p_i,r_i))}{r_i^2}
\\
&\leq C\alpha^2\,\sum_{i\in J}r_i^{n-2}
\leq C\alpha^2\,\sum_{i=1}^\infty r_i^{n-2}<\frac{C\alpha^2}{k}.
\end{align*}
So, we have proved
\begin{equation}
\label{eq:ce5}
\int_\Sg|\nabla_\Sg\zeta_k^J|^2\,da_f<\frac{C\alpha^2}{k}, \ \text{ for any finite set }J\sub\nn.
\end{equation}

Now, we fix a point $p_0\in\Sg_0$, and we denote $\delta_k(p):=1-\mu(d(p,p_0)/k)$. This is a function in $C^\infty_0(\rrn)$ such that $\delta_k=1$ in $\overline{B}(p_0,k/2)$, $\delta_k=0$ in $\rrn-B(p_0,k)$, and $0<\delta_k<1$ in $B(p_0,k)-\overline{B}(p_0,k/2)$. The gradient of $\delta_k$ satisfies $|\nabla_\Sg\delta_k|^2\leq |\nabla\delta_k|^2\leq \alpha^2/k^2$, and so
\begin{equation}
\label{eq:ce7}
\lim_{k\to\infty}\int_\Sg|\nabla_\Sg\delta_k|^2\,da_f=0,
\end{equation}
since $A_f(\Sg)=P_f(E,\Om)<+\infty$.

Finally, we define the function $\var_k:\rrn\to [0,1]$ by $\var_k:=\delta_k\,\zeta^{J(k)}_k$, where $J(k)\sub\nn$ is a finite set such that the compact ball $\overline{B}(p_0,k)$ is contained in the union of the balls $\{B(p_i,r_i/2)\,;\,i\in J(k)\}$ together with the open set $\rrn-\Sg_0$. Clearly, if $p\in\Sg$ and $\var_k(p)\neq 0$, then $p\in\overline{\ptl E\cap\Om}\cap\overline{B}(p_0,k)\cap\big(\rrn-\cup_{i\in J(k)}B(p_i,r_i/2)\big)$. From here, it is easy to deduce that the support of $\var_k$ in $\Sg$ is a compact subset of $\overline{\ptl E\cap\Om}-\Sg_0=\Sg$. Note also that $\var_k\in H^1_0(\Sg,da_f)$. Moreover, by the Cauchy-Schwarz inequality we get
\begin{align*}
\int_\Sg|\nabla_\Sg\var_k|^2\,da_f &\leq\int_\Sg|\nabla_\Sg\delta_k|^2\,da_f+\int_\Sg|\nabla_\Sg\zeta_k^{J(k)}|^2\,da_f
\\
&+2\left(\int_\Sg|\nabla_\Sg\delta_k|^2\,da_f\right)^{1/2}\,
\left(\int_\Sg|\nabla_\Sg\zeta_k^{J(k)}|^2\,da_f\right)^{1/2},
\end{align*}
which tends to zero as $k\to\infty$ by \eqref{eq:ce7} and \eqref{eq:ce5}. On the other hand, it is clear that $\var_k^{-1}(1)=\overline{B}(p_0,k/2)-W'_k$, where $W'_k:=\cup_{i\in J(k)}B(p_i,r_i)$. So, for a given compact set $K\sub\Sg$, we can find by \eqref{eq:dist} a number $k_0\in\nn$ such that $\var_k=1$ on $K$, for any $k\geq k_0$. Therefore, the sequence $\{\var_k\}_{k\in\nn}$ satisfies the desired properties and the proof is completed.
\end{proof} 

\begin{remark}
The main ingredients in the construction of the sequence $\{\var_k\}_{k\in\nn}$ are the perimeter estimates in \eqref{eq:aub} and the condition $\mathcal{H}^{n-2}(\Sg_0)=0$. In our case both ingredients come from the fact that $E$ is a weighted minimizer. By using similar arguments,  Bayle~\cite[Prop.~2.5]{bayle-paper} obtained Theorem~\ref{th:nullcap} inside compact Riemannian manifolds with constant density. In a $C^2$ open set $\Om\subeq\rrn$ with constant density, Sternberg and Zumbrun \cite[Prop.~2.3]{sz} deduced \eqref{eq:aub} from the classical monotonicity formulas for rectifiable varifolds of bounded mean curvature in \cite[Sect.~17]{simon} and \cite{gruter-jost}. This was also done by Morgan and Ritor\'e~\cite[Lem.~3.1]{morgan-rit}, and by Ritor\'e and the author~\cite[Lem.~4.2]{cones}. So, in the Euclidean setting one can find the sequence $\{\var_k\}_{k\in\nn}$ inside any hypersurface $\Sg$ of bounded mean curvature and such that $\Sg_0:=\overline{\Sg}-\Sg$ is a closed set satisfying $\mathcal{H}^{n-2}(\Sg_0)=0$. It is worth mentioning that in all the previous constructions of $\{\var_k\}_{k\in\nn}$ it was assumed that the interior boundary $\overline{\ptl E\cap\Om}$ is compact. So, our Theorem~\ref{th:nullcap} generalizes such constructions to noncompact interior boundaries.
\end{remark}

\section{Characterization of weighted stable sets}
\label{sec:stable}
\setcounter{equation}{0}

In this section we consider the stability condition in Euclidean smooth open sets with smooth densities. By a \emph{weighted stable set} we mean a second order minimum of the weighted perimeter functional for compactly supported variations preserving the weighted volume. Clearly a weighted isoperimetric region is a weighted stable set. As we aim to study the weighted minimizers by means of the stability condition, we can restrict ourselves, by Theorems~\ref{th:reg} and~\ref{th:nullcap}, to stable sets of \emph{finite weighted perimeter} and whose interior boundary coincides, up to a closed set of vanishing area, with a smooth hypersurface of \emph{null weighted capacity}. Under these conditions we will be able to obtain a stability inequality, that we use to classify weighted stable sets in half-spaces and slabs of $\rrn$ endowed with a log-concave perturbation of the Gaussian density as in \eqref{eq:perturb}.

\subsection{Stability inequalities}
\label{subsec:stable1}
\noindent

Let us first introduce some notation and recall the basic variational properties of stable sets. Consider a smooth open set $\Om$ in $\rrn$ endowed with a density $f=e^\psi$ smooth on $\barOm$. Take an open set $E\sub\Om$ whose interior boundary $\overline{\ptl E\cap\Om}$ is the disjoint union of a smooth embedded hypersurface $\Sg$ and a closed singular set $\Sg_0$ with $A_f(\Sg_0)=0$. Note that $\Sg$ is a closed hypersurface of $\rrn$ when $\Sg_0=\emptyset$. If $\Sg$ has non-empty boundary $\ptl\Sg$, then we assume $\ptl\Sg=\Sg\cap\ptl\Om$, which prevents the existence in $\ptl\Om$ of interior points of $\Sg$. If $\ptl\Sg=\emptyset$ then we adopt the convention that all the integrals along $\ptl\Sg$ vanish. We denote by $N$ the inner unit normal of $\Sg$, and by $\nu$ the conormal vector of $\ptl\Sg$, i.e., the inner unit normal along $\ptl\Sg$ in $\Sg$.

Let $X$ be a smooth vector field on $\rrn$ with compact support on $\overline{\Om}$ and tangent along $\ptl\Om$. The one-parameter group of diffeomorphisms $\{\phi_s\}_{s\in\rr}$ of $X$ allows to define a \emph{variation} of $E$ by $E_s:=\phi_s(E)$. The associated volume and perimeter functionals are given by $V_f(s):=V_f(E_s)$ and $P_f(s):=P_f(E_s,\Om)$, respectively. The variation is said to be \emph{volume-preserving} if $V_f(s)=V_f(E)$ for any $s$ small enough. We say that $E$ is \emph{weighted stationary} if $P_f'(0)=0$ for any volume-preserving variation. We say that $E$ is \emph{weighted stable} if it is weighted stationary and $P_f''(0)\geq 0$ for any volume-preserving variation. If we denote $\Sg_s:=\phi_s(\Sg)$, then $\ptl\Sg_s=\Sg_s\cap\ptl\Om$ and we know that $P_f(s)=A_f(\Sg_s)$. Thus, if $E$ is weighted stable, then the hypersurface $\Sg$ is \emph{free boundary $f$-stable} in the sense defined by Castro and the author \cite[Sect.~3]{castro-rosales}, i.e., the area functional $A_f(s):=A_f(\Sg_s)$ satisfies $A_f'(0)=0$ and $A_f''(0)\geq 0$ for any volume-preserving variation. Indeed, if $\Sg_0=\emptyset$, then $E$ is weighted stable if and only if $\Sg$ is $f$-stable. By using the first and second variation formulas for $V_f(s)$ and $A_f(s)$ we deduce the following result, see \cite[Sect.~3]{castro-rosales} for details.

\begin{lemma}
\label{lem:varprop}
Let $\Om$ be a smooth open set of $\rrn$ endowed with a density $f=e^\psi$ smooth on $\barOm$. Consider an open set $E\sub\Om$ such that $\overline{\ptl E\cap\Om}=\Sg\cup\Sg_0$, where $\Sg$ is a smooth hypersurface with boundary $\ptl\Sg=\Sg\cap\ptl\Om$, and $\Sg_0$ is a closed singular set with $A_f(\Sg_0)=0$.
\begin{itemize}
\item[(i)] If $E$ is weighted stationary, then $\Sg$ has constant $f$-mean curvature $($defined in \eqref{eq:fmc}$)$ with respect to the inner unit normal $N$, and $\Sg$ meets $\ptl\Om$ orthogonally along $\ptl\Sg$.
\item[(ii)] If $E$ is weighted stable, then the $f$-index form of $\Sg$ $($defined below$)$ satisfies $\indo_f(u,u)\geq 0$ for any $u\in C^\infty_0(\Sg)$ with $\int_\Sg u\,da_f=0$.
\end{itemize}
Moreover, if $\Sg_0=\emptyset$, then the reverse statements in \emph{(i)} and \emph{(ii)} hold.
\end{lemma}

The \emph{$f$-index form} of $\Sg$ is the symmetric bilinear form on $C_0^\infty(\Sg)$ given by
\begin{equation}
\label{eq:index1}
\indo_f(u,v):=\int_\Sg\left\{\escpr{\nabla_\Sg u,\nabla_\Sg v}-\big(\ric_f(N,N)+|\sigma|^{2}\big)\,uv\right\}da_{f}
-\int_{\ptl\Sg}\text{II}(N,N)\,uv\,dl_f,
\end{equation}
where $\ric_f$ denotes the Bakry-\'Emery-Ricci tensor defined in \eqref{eq:fricci},  $|\sg|^2$ is the squared norm of the second fundamental form of $\Sg$, and $\text{II}$ is the second fundamental form of $\ptl\Om$ with respect to the inner unit normal. From the integration by parts formula in \eqref{eq:ibp}, we get 
\[
\mathcal{I}_f(u,v)=\mathcal{Q}_f(u,v), \ \text{ for any } u,v\in C^\infty_0(\Sg),
\]
where 
\begin{equation}
\label{eq:index2}
\mathcal{Q}_f(u,v):=-\int_\Sg u\,\mathcal{L}_f(v)\,da_f
-\int_{\ptl\Sg}u\,\left\{\frac{\ptl v}{\ptl\nu}+\text{II}(N,N)\,v\right\}dl_f.
\end{equation}
In the previous equation, $\mathcal{L}_f$ is the \emph{$f$-Jacobi operator of $\Sg$}, i.e., the second order linear operator given by
\begin{equation}
\label{eq:jacobi}
\mathcal{L}_f(v):=\Delta_{\Sg,f}\,v+\left(\text{Ric}_f(N,N)+|\sg|^2\right)v,
\end{equation}
where $\Delta_{\Sg,f}$ is the $f$-Laplacian relative to $\Sg$ in \eqref{eq:deltaf}. It is important to recall that $\mathcal{L}_f(v)$ coincides with the derivative of the $f$-mean curvature along a variation whose velocity vector $X_p:=(d/ds)|_{s=0}\,\phi_s(p)$ satisfies $\escpr{X,N}=v$ along $\Sg$, see \cite[Eq.~(3.5)]{castro-rosales}. This means that
\begin{equation}
\label{eq:hfprima}
(\mathcal{L}_f (v))(p)=\frac{d}{ds}\bigg|_{s=0}(H_f)_s(\phi_s(p)), 
\ \text{ for any } p\in\Sg,
\end{equation}
where $(H_f)_s$ denotes the $f$-mean curvature along the hypersurface $\Sg_s$. If we further assume that $\Sg$ has null weighted capacity, then the integration by parts formula in Lemma~\ref{lem:ibpext} and the symmetry of the $f$-index form $\indo_f$ allow us to prove the following: for any two bounded functions $u_i\in C^\infty(\Sg)\cap H^1(\Sg,da_f)$ such that $\Delta_{\Sg,f}u_i\in L^1(\Sg,da_f)$ and $\ptl u_i/\ptl\nu\in L^1(\ptl\Sg,dl_f)$, we have $\ind_f(u_1,u_2)=\ind_f(u_2,u_1)$. This is equivalent to the identity
\begin{equation}
\label{eq:symmetry}
\int_\Sg \left\{u_1\,\mathcal{L}_f(u_2)-u_2\,\mathcal{L}_f(u_1)\right\}da_f
=\int_{\ptl\Sg}\left\{u_2\,\frac{\ptl u_1}{\ptl\nu}-u_1\,\frac{\ptl u_2}{\ptl\nu}\right\}dl_f,
\end{equation} 
that will be used in the proof of Theorem~\ref{th:stable}.

Observe that the stability inequality in Lemma~\ref{lem:varprop} (ii) is valid for mean zero functions \emph{with compact support on} $\Sg$. Geometrically these functions come from volume-preserving variations of $E$ fixing a neighborhood of the singular set $\Sg_0$. Note also that the stability inequality becomes more restrictive provided the Bakry-\'Emery-Ricci curvature $\ric_f$ is nonnegative and $\Om$ is convex. By assuming these hypotheses we obtain below more general inequalities for mean zero functions satisfying certain integrability conditions. Their proofs rely on approximation arguments, that can be carried out when $\Sg$ has null weighted capacity.

\begin{proposition}[Generalized stability inequalities]
\label{prop:extended}
Let $\Om$ be a smooth convex domain of $\rrn$ endowed with a density $f=e^\psi$, smooth on $\barOm$, and satisfying $\emph{Ric}_f\geq 0$.  Consider an open set $E\sub\Om$ of finite weighted perimeter in $\Om$ such that $\overline{\ptl E\cap\Om}=\Sg\cup\Sg_0$, where $\Sg$ is a smooth hypersurface of null weighted capacity with boundary $\ptl\Sg=\Sg\cap\ptl\Om$, and $\Sg_0$ is a closed singular set with $A_f(\Sg_0)=0$. If $E$ is weighted stable, then
\begin{itemize}
\item[(i)] For any bounded function $u\in H^1(\Sg,da_f)$ with $\int_\Sg u\,da_f=0$, we have $(\emph{Ric}_f(N,N)+|\sg|^2)\,u^2\in L^1(\Sg,da_f)$, $\emph{II}(N,N)\,u^2\in L^1(\ptl\Sg,dl_f)$, and 
\[
\indo_f(u,u)\geq 0,
\]
where $\indo_f$ is the $f$-index form defined in \eqref{eq:index1}. Moreover, if $\Sg_0=\emptyset$, then the same holds without assuming that $u$ is bounded.
\item[(ii)] For any bounded function $u\in C^\infty(\Sg)\cap H^1(\Sg,da_f)$ with $\Delta_{\Sg,f}u\in L^1(\Sg,da_f)$, $\ptl u/\ptl\nu\in L^1(\ptl\Sg,dl_f)$ and $\int_\Sg u\,da_f=0$, we have $(\emph{Ric}_f(N,N)+|\sg|^2)\,u^2\in L^1(\Sg,da_f)$, $\emph{II}(N,N)\,u^2\in L^1(\ptl\Sg,dl_f)$, and 
\[
\ind_f(u,u)\geq 0,
\]
where $\ind_f$ is the $f$-index form defined in \eqref{eq:index2}. 
\end{itemize}
Furthermore, it also holds $\emph{Ric}_f(N,N)+|\sg|^2\in L^1(\Sg,da_f)$ and $\emph{II}(N,N)\in L^1(\ptl\Sg,dl_f)$.
\end{proposition}

\begin{proof}
We follow the proof given by Ritor\'e and the author for bounded stable sets in convex solid cones with constant densities,  see \cite[Lems.~4.5 and 4.7]{cones}. 

Note that $A_f(\Sg)=P_f(E,\Om)<+\infty$ since $E$ has almost smooth interior boundary. As $\text{Cap}_f(\Sg)=0$, we can apply Lemma~\ref{lem:capacity} to find a sequence $\{\var_k\}_{k\in\nn}\sub C^\infty_0(\Sg)$ such that $0\leq\var_k\leq 1$, $\lim_{k\to\infty}\var_k(p)=1$ for any $p\in\Sg$, and $\lim_{k\to\infty}\int_\Sg|\nabla_\Sg\var_k|^2\,da_f=0$. Take a bounded function $u\in C^\infty(\Sg)\cap H^1(\Sg,da_f)$ with $\int_\Sg u\,da_f=0$. Define $u_k:=\var_k u-\alpha_k \varphi$, where $\var\in C^\infty_0(\Sg)$ is a function with $\int_\Sg\var\,da_f=1$ and $\alpha_k:=\int_\Sg\var_k u\,da_f$. This gives us a sequence $\{u_k\}_{k\in\nn}\sub C^\infty_0(\Sg)$ such that $\int_\Sg u_k\,da_f=0$ and $\lim_{k\to\infty}u_k(p)=u(p)$ for any $p\in\Sg$ (note that $\lim_{k\to\infty}\alpha_k=\int_\Sg u\,da_f=0$). Moreover, from the dominated convergence theorem and the Cauchy-Schwarz inequality, we obtain $\lim_{k\to\infty}\int_\Sg|\nabla_\Sg u_k|^2\,da_f=\int_\Sg|\nabla_\Sg u|^2\,da_f$. Since $E$ is weighted stable, we deduce from Lemma~\ref{lem:varprop} (ii) that $\indo_f(u_k,u_k)\geq 0$ for any $k\in\nn$. This means that
\[
\int_\Sg\big(\ric_f(N,N)+|\sg|^2\big)\,u_k^2\,da_f
+\int_{\ptl\Sg}\text{II}(N,N)\,u_k^2\,da_f\leq
\int_\Sg|\nabla_\Sg u_k|^2\,da_f
\] 
for any $k\in\nn$. As $\ric_f(N,N)\geq 0$ and $\text{II}(N,N)\geq 0$ (by convexity of $\Om$), we can apply Fatou's lemma and that $\int_\Sg|\nabla_\Sg u|^2\,da_f<+\infty$ to get $(\ric_f(N,N)+|\sg|^2)\,u^2\in L^1(\Sg,da_f)$, $\text{II}(N,N)\,u^2\in L^1(\ptl\Sg,dl_f)$ and $\indo_f(u,u)\geq 0$. Thus, we have shown that $\indo_f(u,u)\geq 0$, for any bounded function $u\in C^\infty(\Sg)\cap H^1(\Sg,da_f)$ with $\int_\Sg u\,da_f=0$. Moreover, if $\Sg_0=\emptyset$, then $\Sg$ is complete, and this allows to construct as in Example~\ref{ex:finite} a sequence $\{\var_k\}_{k\in\nn}\sub C^\infty_0(\Sg)$ satisfying also $|\nabla_\Sg\var_k|\leq\alpha/k$ in $\Sg$ for some constant $\alpha>0$. From this gradient estimate we can reproduce the previous arguments even if $u$ is not assumed to be bounded.  Statement (i) then follows from a standard regularization argument. If we further assume $\Delta_{\Sg,f}u\in L^1(\Sg,da_f)$ and $\ptl u/\ptl\nu\in L^1(\ptl\Sg,dl_f)$, then the integration by parts formula for hypersurfaces of null weighted capacity in Lemma~\ref{lem:ibpext} implies that $\ind_f(u,u)=\indo_f(u,u)\geq 0$. This proves (ii).

To finish the proof it suffices to see that $\ric_f(N,N)+|\sg|^2$ and $\text{II}(N,N)$ are integrable over the complement $\Sg-B$ of a small geodesic ball $B\sub\sub\Sg-\ptl\Sg$. As $A_f(\Sg)<+\infty$, we can take a bounded function $u\in C^\infty(\Sg)$ such that $\int_\Sg u\,da_f=0$ and $u=1$ in $\Sg-B$. From statement (i) we conclude that $(\ric_f(N,N)+|\sg|^2)\,u^2\in L^1(\Sg,da_f)$ and $\text{II}(N,N)\,u^2\in L^1(\ptl\Sg,dl_f)$. The claim follows by using that $u=1$ in $\Sg-B$.
\end{proof}

To finish this section we apply Lemma~\ref{lem:varprop} and Proposition~\ref{prop:extended} to analyze when Euclidean half-spaces are weighted stationary or stable for the perturbations of the Gaussian density defined in \eqref{eq:perturb}.

\begin{lemma}
\label{lem:half-spaces}
Consider an open half-space or slab $\Om:=\rr^n\times (a,b)$ endowed with the density $f=e^\psi$, where $\psi(p):=\omega(\pi(p))-c |p|^2$ for some function $\omega$ smooth in the closure of $(a,b)$ and some $c>0$. Let $E$ be the intersection with $\Om$ of an open half-space in $\rrn$ with interior boundary $\Sg:=\overline{\ptl E\cap\Om}$.
\begin{itemize}
\item[(i)] If $\Sg$ is either parallel or perpendicular to $\ptl\Om$, then $\Sg$ has constant $f$-mean curvature. Otherwise, $\Sg$ has constant $f$-mean curvature if and only if there are constants $\alpha,\beta\in\rr$ such that $\omega(t)=\alpha\,t+\beta$, for any $t\in (a,b)$.
\item[(ii)] $E$ is weighted stationary in $\Om$ if and only if $\Sg$ is parallel or perpendicular to $\ptl\Om$.
\item[(iii)] If $\Sg$ is parallel to $\ptl\Om$ with $\pi(p)=t_0$ in $\Sg$, then $E$ is weighted stable if and only if $\omega''(t_0)\geq 0$.
\end{itemize}
\end{lemma}

\begin{proof}
Let $N$ be the inner unit normal to $\Sg$. By taking into account \eqref{eq:fmc} and the first equality in \eqref{eq:riccperturb}, we get
\[
H_f(p)=-\omega'(\pi(p))\,\escpr{\ptl_t,N_p}+2c\,\escpr{p,N_p}, \ \text{ for any }p\in\Sg,
\]
from which we deduce (i). Statement (ii) follows from (i) and the orthogonality condition in Lemma~\ref{lem:varprop} (i). 

Now we prove (iii). Take a function $u\in C^\infty_0(\Sg)$ with $0=\int_\Sg u\,da_f=e^{\omega(t_0)}\int_\Sg u\,da_c$, where $da_c$ denotes the weighted area measure associated to the Gaussian density $\ga_c(p):=e^{-c |p|^2}$. From the definition of $f$-index form in \eqref{eq:index1}, the second equality in \eqref{eq:riccperturb}, and the fact that $\pi(p)=t_0$ in $\Sg$, we obtain
\[
\indo_f(u,u)=e^{\omega(t_0)}\left [\int_\Sg\big(|\nabla_\Sg u|^2-2c u^2\big)\,da_c+\omega''(t_0)\int_\Sg u^2\,da_c\right].
\]
On the other hand, the Gaussian isoperimetric inequality implies that $\Sg$ bounds a Gaussian minimizer and, in particular, a weighted stable set for $\ga_c$. Therefore, we have the Poincar\'e type inequality
\[
\int_\Sg\big(|\nabla_\Sg u|^2-2c u^2\big)\,da_c\geq 0,
\]
and equality holds for any coordinate function $u$ on $\Sg$. Hence it is clear from Lemma~\ref{lem:varprop} (ii) that $E$ is weighted stable if $\omega''(t_0)\geq 0$. Conversely, suppose that $E$ is weighted stable and $\omega''(t_0)<0$. This gives us $\ric_f\geq 0$ on $\Sg$ by the second equality in \eqref{eq:riccperturb}. So,  we may apply the stability inequality in Proposition~\ref{prop:extended} (i) with a coordinate function $u$ in order to get $\indo_f(u,u)<0$, a contradiction. This proves the claim.
\end{proof}

\begin{remarks}
\label{re:sanasana}
1. The previous result is also valid when $\Om=\rrn$. In this case $\ptl\Om=\emptyset$ and $E$ is weighted stationary if and only if $\Sg$ has constant $f$-mean curvature. We stress that weighted stationary half-spaces which are neither horizontal nor vertical appear if and only if $\omega:\rr\to\rr$ is an affine function.

2. From Lemma~\ref{lem:half-spaces} (iii) we deduce that, if $\omega$ is strictly concave, then half-spaces parallel to $\ptl\Om$ are weighted unstable sets. In particular, they cannot be weighted minimizers.

3. Note that the weighted stability of half-spaces perpendicular to $\ptl\Om$ is not discussed in Lemma~\ref{lem:half-spaces}. When $\omega$ is concave we will show, as an immediate consequence of Theorem~\ref{th:isoperimetric}, that they are all weighted stable. 

4. Barthe, Bianchini and Colesanti have employed Poincar\'e type inequalities to study the weighted stability of half-spaces in $\rrn$ for a Euclidean measure $\mu^{n+1}$, where $\mu$ is a probability measure on $\rr$, see \cite[Sect.~3]{chiara}. In particular, they obtained stability of half-spaces in $\rrn$ with Gaussian density. Indeed, it is proved in \cite[Thm.~3.5]{chiara} that the weighted stability of coordinate half-spaces characterizes Gaussian type measures. By using similar arguments, Doan provides in \cite[Thm.~5.1]{calibrations} a stability criterion for horizontal half-spaces in $\rrn=\rr^n\times\rr$ with product density $f(z,t)=e^{\psi(z)+\delta(t)}$. In fact, a calibration argument shows that, if $\delta$ is constant or strictly convex, then any horizontal half-space is weighted area-minimizing, see \cite[Thm.~2.9]{cvm} and \cite[Cor.~5.3]{calibrations}.
\end{remarks}

\subsection{Characterization results}
\label{subsec:stable2}
\noindent

We now turn to the classification of weighted stable sets in a half-space or slab $\Om:=\rr^n\times (a,b)$ endowed with a smooth log-concave perturbation of the Gaussian density as in \eqref{eq:perturb}. For that, we will use the stability inequality in Proposition~\ref{prop:extended} (ii) with a suitable test function. Since $\ptl\Om$ is totally geodesic it is natural to produce such a function by using translations along $\ptl\Om$. However, as a difference with respect to the case of constant density, such translations need not preserve the enclosed volume. To solve this difficulty we consider variations of a given stable set $E$ by equidistant sets translated along $\ptl\Om$ to keep the weighted volume constant. The associated test functions along the interior boundary of $E$ are computed in the next lemma.   

\begin{lemma}[Test functions]
\label{lem:test}
Let $\Om$ be an open half-space or slab in $\rrn$ endowed with a density $f=e^\psi$ smooth on $\barOm$. Consider a bounded open set $E\sub\Om$ such that $\overline{\ptl E\cap\Om}$ is a smooth hypersurface $\Sg$ meeting orthogonally $\ptl\Om$ in the points of $\ptl\Sg=\Sg\cap\ptl\Om$. For any unit vector $\eta$ tangent to $\ptl\Om$ we consider the variation of $\Sg$ defined by
\[
\phi_s(p):=p+sN_p+\la(s)\eta, 
\]
where $\la(s)$ is a smooth function such that $\la(0)=0$ and the variation is volume-preserving. Then, the normal component of the associated velocity vector is the function
\[
v(p):=\frac{\alpha+\escpr{\eta,N_p}}{\alpha}, \ \text{ for any }p\in\Sg,
\]
where $\alpha:=-A_f(\Sg)^{-1}\,\int_\Sg\escpr{\eta,N_p}\,da_f$.
\end{lemma}

\begin{proof}
The associated velocity vector of $\{\phi_s\}_{s\in\rr}$ is $X_p:=(d/ds)|_{s=0}\,\phi_s(p)=N_p+\la'(0)\eta$, which is tangent to $\ptl\Om$ in the points of $\ptl\Sg$. Thus we have $v(p):=\escpr{X_p,N_p}=1+\la'(0)\,\escpr{\eta,N_p}$, for any $p\in\Sg$. Let $E_s\sub\Om$ be the bounded open set such that $\overline{\ptl E_s\cap\Om}=\Sg_s$. By using that the function $V_f(s):=V_f(E_s)$ is constant for $s$ small enough, and the first variation of weighted volume computed in \cite[Sect.~3]{castro-rosales}, we deduce
\[
0=V_f'(0)=-\int_\Sg v\,da_f=-A_f(\Sg)-\la'(0)\int_\Sg\escpr{\eta,N_p}\,da_f.
\]
Therefore $\la'(0)=\alpha^{-1}$ and the proof finishes.
\end{proof}

For a half-space or slab $\Om:=\rr^n\times (a,b)$ with a smooth density $f$ as in \eqref{eq:perturb} we discussed in Lemma~\ref{lem:half-spaces} when half-spaces intersected with $\Om$ are weighted stationary or stable. In particular, we showed that the stability of half-spaces parallel to $\ptl\Om$ becomes more restrictive when the function $\omega$ is concave. On the other hand, the concavity of $\omega$ is also equivalent by \eqref{eq:riccperturb} to the inequality $\ric_f\geq 2c$, which allows to deduce the stability inequalities in Proposition~\ref{prop:extended}. In the main result of this section we assume concavity of $\omega$ to characterize stable sets in $\Om$ of finite weighted perimeter and interior boundary of null weighted capacity.
 
\begin{theorem}[Stable sets in half-spaces and slabs]
\label{th:stable}
Consider an open half-space or slab $\Om:=\rr^n\times (a,b)$ endowed with the density $f=e^\psi$, where $\psi(p):=\omega(\pi(p))-c |p|^2$ for some concave function smooth in the closure of $(a,b)$ and some $c>0$. Let $E\sub\Om$ be an open set of finite weighted perimeter in $\Om$ such that $\overline{\ptl E\cap\Om}=\Sg\cup\Sg_0$, where $\Sg$ is a smooth hypersurface of null weighted capacity with boundary $\ptl\Sg=\Sg\cap\ptl\Om$, and $\Sg_0$ is a closed singular set with $A_f(\Sg_0)=0$. If $E$ is weighted stable, then $E$ is the intersection with $\Om$ of a half-space with boundary parallel or perpendicular to $\ptl\Om$. Moreover, if $\omega$ is strictly concave, only the latter case is possible.
\end{theorem}

\begin{proof}
Let $N$ be the inner unit normal to $\Sg$. From Lemma~\ref{lem:varprop} (i) we know that $\Sg$ has constant $f$-mean curvature and meets $\ptl\Om$ orthogonally in the points of $\ptl\Sg$. Hence, along $\ptl\Sg$, the conormal vector $\nu$ coincides with the inner unit normal to $\ptl\Om$. Note that $\ric_f\geq 0$ by the second equality in \eqref{eq:riccperturb}. Hence the weighted stability of $E$ implies that $\Sg$ is connected: otherwise, we would deduce $\indo(u,u)<0$ for some locally constant function $u$, which contradicts Proposition~\ref{prop:extended} (i).

For a fixed unit horizontal vector $\eta$, we define the vector field $X_p:=\eta$ with associated one-parameter group of translations $\tau_s(p):=p+s\eta$. Let $\Sg_s:=\tau_s(\Sg)$, and denote by $(H_f)_s$ the $f$-mean curvature of $\Sg_s$. By using \eqref{eq:fmc}, \eqref{eq:riccperturb}, equality $\pi(\tau_s(p))=\pi(p)$, and that the unit normal $N_s$ to $\Sg_s$ and the Euclidean mean curvature $H_s$ of $\Sg_s$ are invariant under translations, we get
\begin{align*}
(H_f)_s(\tau_s(p))&=\big(nH_s-\escpr{\nabla\psi,N_s}\big)(\tau_s(p))
\\
&=nH(p)-\omega'(\pi(p))\,\escpr{\ptl_t,N_p}+2c\,\big(\escpr{p,N_p}+s\,\escpr{\eta,N_p}\big).
\end{align*}
Hence, if we define $h:\Sg\to\rr$ by $h(p):=\escpr{X_p,N_p}=\escpr{\eta,N_p}$, then equations \eqref{eq:jacobi} and \eqref{eq:hfprima} give us
\begin{equation}
\label{eq:eigen}
\Delta_{\Sg,f}h+(\ric_f(N,N)+|\sg|^2)\,h=\mathcal{L}_f (h)=\frac{d}{ds}\bigg|_{s=0}(H_f)_s(\tau_s(p))=2c h,
\end{equation}
where $\Delta_{\Sg,f}$ is the $f$-Laplacian relative to $\Sg$ and $\sg$ is the second fundamental form of $\Sg$.  Formula~\eqref{eq:eigen} shows that $h$ is an eigenfunction for the $f$-Jacobi operator on $\Sg$.

Observe that $h\in L^1(\Sg,da_f)$ since $h$ is bounded and $A_f(\Sg)=P_f(E,\Om)<+\infty$. Thus, we can define the test function
\begin{equation*}
u:=\alpha+h=\alpha\,v,
\end{equation*} 
where  $\alpha:=-A_f(\Sg)^{-1}\int_\Sg h\,da_f$ and $v$ is the function in Lemma~\ref{lem:test}. Note that $u$ is a bounded smooth function on $\Sg$ with $\int_\Sg u\,da_f=0$. Let us see that $u$ satisfies the integrability hypotheses in Proposition~\ref{prop:extended} (ii). From the definition of $h$ and equation \eqref{eq:eigen}, we obtain
\begin{align}
\label{eq:mika1}
(\nabla_\Sg u)_p&=(\nabla_\Sg h)_p=-\sum_{i=1}^nk_i(p)\,\escpr{\eta,e_i}\,e_i,
\\ 
\nonumber
\Delta_{\Sg,f} u&=\Delta_{\Sg,f}h=2c h-(\ric_f(N,N)+|\sg|^2)\,h,
\end{align}
where $e_i$ is a principal direction of $\Sg$ at $p$ with principal curvature $k_i(p)$. Note that $|\nabla_\Sg u|^2\leq |\sg|^2\leq \ric_f(N,N)+|\sg|^2$, which is an integrable function by Proposition~\ref{prop:extended}. Hence, it follows that $u\in H^1(\Sg,da_f)$ and $\Delta_{\Sg,f}u\in L^1(\Sg,da_f)$. Now, we compute $\ptl u/\ptl\nu$ along $\ptl\Sg$. Let $D$ be the Levi-Civit\`a connection in $\rrn$. For any vector $T$ tangent to $\ptl\Sg$, we have
\[
-\escpr{D_\nu N,T}=\sigma(\nu,T)=\sigma(T,\nu)=-\escpr{D_T N,\nu}
=\escpr{N,D_T\nu}=0,
\]
since $\ptl\Om$ is totally geodesic and $\nu$ coincides with the inner unit normal to $\ptl\Om$. The previous computation shows that $D_\nu N$ is proportional to $\nu$ in the points of $\ptl\Sg$. By taking an orthonormal basis $\{e_1,\ldots, e_n\}$ of principal directions at $p\in\ptl\Sg$ with $e_n=\nu_p$, and having in mind \eqref{eq:mika1}, we deduce
\begin{equation}
\label{eq:bdterm}
\frac{\ptl u}{\ptl\nu}\,(p)=\frac{\ptl h}{\ptl\nu}\,(p)
=-k_n(p)\,\escpr{\eta,\nu_p}=0,
\end{equation}
since $\eta$ is tangent to $\ptl\Om$ and $\nu_p$ is normal to $\ptl\Om$ at $p$. 

At this point, we can apply Proposition~\ref{prop:extended} (ii) to ensure that $\ind_f(u,u)\geq 0$, where $\ind_f$ is the $f$-index form of $\Sg$ defined in \eqref{eq:index2}. Observe that the boundary term in $\ind_f(u,u)$ vanishes since $\ptl u/\ptl\nu=\text{II}(N,N)=0$ along $\ptl\Sg$. By using \eqref{eq:eigen}, the second equality in \eqref{eq:riccperturb}, and that $\int_\Sg u\,da_f=0$, we have
\begin{align}
\label{eq:memo}
0\leq\ind_f(u,u)&=-\int_\Sg u\,\mathcal{L}_f (u)\,da_f=-\int_\Sg (\alpha+h)\,\mathcal{L}_f(\alpha+h)\,da_f
\\
\nonumber
&=-\int_\Sg (\alpha+h)\,\{\alpha\,\big(2c-\omega''(\pi(p))\,\escpr{\ptl_t,N_p}^2+|\sg|^2\big)+2c h\}\,da_f
\\
\nonumber
&=-\int_\Sg (\alpha+h)\,\{\alpha\,\big(|\sg|^2-\omega''(\pi(p))\,\escpr{\ptl_t,N_p}^2\big)+2c h\}\,da_f.
\end{align}
On the other hand, equalities \eqref{eq:eigen}, \eqref{eq:symmetry} and \eqref{eq:bdterm} give us
\begin{align*}
\int_\Sg 2c h\,da_f&=\int_\Sg\mathcal{L}_f(h)\,da_f
=\int_\Sg h\,\mathcal{L}_f(1)\,da_f
\\
&=\int_\Sg h\,\big(2c-\omega''(\pi(p))\,\escpr{\ptl_t,N_p}^2+|\sg|^2\big)\,da_f,
\end{align*}
and so
\[ 
\int_\Sg h\,\big(|\sg|^2-\omega''(\pi(p))\,\escpr{\ptl_t,N_p}^2\big)\,da_f=0. 
\]
Plugging this into \eqref{eq:memo}, and taking into account that $\alpha=-A_f(\Sg)^{-1}\int_\Sg h\,da_f$, we get
\begin{align*}
0&\leq-\,\int_\Sg \alpha^2\,\big(|\sg|^2-\omega''(\pi(p))\,\escpr{\ptl_t,N_p}^2\big)\,da_f
\\
&+2c\left[\frac{1}{A_f(\Sg)}\left(\int_\Sg h\,da_f\right)^2-\int_\Sg h^2\,da_f\right]\leq 0,
\end{align*}
by the concavity of $\omega$ and the Cauchy-Schwarz inequality in $L^2(\Sg,da_f)$. As a consequence, we can ensure that $\alpha^2\,(|\sg|^2-\omega''(\pi(p))\,\escpr{\ptl_t,N_p}^2)=0$ on $\Sg$ and $\escpr{\eta,N_p}=h(p)=-\alpha$, for any $p\in\Sg$ and any unit horizontal vector $\eta$. If $\alpha=0$ for any $\eta$, then $\Sg$ is contained in a hyperplane parallel to $\ptl\Om$. If $\alpha\neq 0$ for some $\eta$, then $|\sg|^2=0$ and so $\Sg$ is contained in a hyperplane transversal to $\ptl\Om$. In both cases, the unit normal $N$ to $\Sg$ extends continuously to $\overline{\Sg}$. This implies by Lemma~\ref{lem:giusti} that $\ptl E\cap\Om$ is a $C^1$ hypersurface since the generalized unit normal $N_f$ in \eqref{eq:gg2} equals $N$ along $\Sg$. In particular $\Sg_0=\emptyset$. As $\Sg$ is closed and connected, we deduce that $\Sg$ is a hyperplane intersected with $\Om$. From the orthogonality condition between $\Sg$ and $\ptl\Om$ in Lemma~\ref{lem:varprop} (i) we conclude that $\Sg$ is perpendicular to $\ptl\Om$ if $\ptl\Sg\neq\emptyset$. Otherwise, $\Sg$ is parallel to $\ptl\Om$. Moreover, if $\omega$ is strictly concave, then $\Sg$ must be perpendicular to $\ptl\Om$ by Lemma~\ref{lem:half-spaces} (iii). This completes the proof of the theorem.
\end{proof}

As a particular case of Theorem~\ref{th:stable} we obtain the following result for the Gaussian density, which is interesting in itself.

\begin{corollary}[Stable sets in Gaussian half-spaces and slabs]
\label{cor:stableGausshalf}
Let $\Om$ be an open half-space or slab in $\rrn$ with Gaussian density $\ga_c(p):=e^{-c |p|^2}$, $c>0$. Consider an open set $E\sub\Om$ of finite weighted perimeter in $\Om$ such that $\overline{\ptl E\cap\Om}=\Sg\cup\Sg_0$, where $\Sg$ is a smooth hypersurface of null weighted capacity with boundary $\ptl\Sg=\Sg\cap\ptl\Om$, and $\Sg_0$ is a closed singular set with $A_f(\Sg_0)=0$. If $E$ is weighted stable, then $E$ is the intersection with $\Om$ of a half-space with boundary parallel or perpendicular to $\ptl\Om$.
\end{corollary}

The arguments in the proof of Theorem~\ref{th:stable} also apply to show that a stable set of finite weighted perimeter and null weighted capacity in $\rrn$ with a smooth log-concave perturbation of the Gaussian density as in \eqref{eq:perturb} must be a half-space. In this case the computations are even easier since $\ptl\Sg=\emptyset$. By combining this fact with Lemma~\ref{lem:half-spaces}, we get

\begin{theorem}[Stable sets in $\rrn$]
\label{th:stableGausswhole}
Consider the density $f=e^\psi$, where $\psi(p):=\omega(\pi(p))-c |p|^2$ for some concave function $\omega\in C^\infty(\rr)$ and some $c>0$. Let $E\sub\rrn$ be an open set of finite weighted perimeter such that $\ptl E=\Sg\cup\Sg_0$, where $\Sg$ is a smooth hypersurface of null weighted capacity, and $\Sg_0$ is a closed singular set with $A_f(\Sg_0)=0$. If $E$ is weighted stable, then $E$ is a half-space. Moreover:
\begin{itemize}
\item[(i)] if $\omega$ is not an affine function, then $\ptl E$ is horizontal or vertical,
\item[(ii)] if $\omega$ is strictly concave, then $\ptl E$ is vertical.
\end{itemize}
\end{theorem}

As the particular case of Theorem~\ref{th:stableGausswhole} when $\omega=0$ we deduce a characterization result for weighted stable sets in the Gauss space. This will be used in Section~\ref{sec:main} to provide a new proof of the Gaussian isoperimetric inequality.

\begin{corollary}[Stable sets in the Gauss space]
\label{cor:wholeGauss}
Let $E$ be an open set of finite weighted perimeter in $\rrn$ with Gaussian density $\ga_c(p):=e^{-c |p|^2}$, $c>0$. Suppose that $\ptl E=\Sg\cup\Sg_0$, where $\Sg$ is a smooth hypersurface of null weighted capacity and $\Sg_0$ is a closed singular set with $A_f(\Sg_0)=0$. If $E$ is weighted stable, then $E$ is a half-space.
\end{corollary}

\begin{remark}
In our stability results the technical hypothesis $\text{Cap}_f(\Sg)=0$ cannot be removed. As we mentioned in Section~\ref{subsec:capacity} this condition implies that the singular set $\Sg_0$ is negligible and we can work as in the case $\Sg$ compact. However, when $\Sg_0$ is not small enough, then it is possible to find weighted stable sets such that the regular part of the boundary is not totally geodesic. Let us illustrate this in the most simple situation of the Gauss plane $(\rr^2,\ga_c)$. It is clear that the weighted length of a curve in $(\rr^2,\ga_c)$ coincides with its Riemannian length in $(\rr^2,g)$, where $g$ is the conformal metric given by $\ga_c\,\escpr{\cdot\,,\cdot}$. Take any open polygon $E\sub\rr^2$ bounded by length-minimizing geodesics in $(\rr^2,g)$. By construction this set is a second order minimum of the weighted perimeter for arbitrary variations compactly supported on $\Sg$. On the other hand, since a segment has vanishing curvature in $(\rr^2,g)$ if and only if it is part of a straight line containing the origin, we conclude that $\Sg$ is not totally geodesic in $\rr^2$. 
\end{remark}

As we showed in Example~\ref{ex:finite}, any complete hypersurface of finite weighted area has null weighted capacity. Hence, by using the same technique as in Theorem~\ref{th:stable}, we obtain the following classification result for free boundary $f$-stable hypersurfaces.

\begin{corollary}[Complete $f$-stable hypersurfaces of finite weighted area]
\label{cor:stableGauss}
Consider an open half-space or slab $\Om:=\rr^n\times (a,b)$ endowed with the density $f=e^\psi$, where $\psi(p):=\omega(\pi(p))-c |p|^2$ for some concave function $\omega$ smooth on the closure of $(a,b)$ and some $c>0$. Let $\Sg\sub\overline{\Om}$ be a smooth, complete, orientable hypersurface with boundary $\ptl\Sg=\Sg\cap\ptl\Om$. If $\Sg$ is $f$-stable and has finite weighted area, then $\Sg$ is either a hyperplane parallel to $\ptl\Om$ or the intersection with $\Om$ of a hyperplane perpendicular to $\ptl\Om$.  Moreover, if $\omega$ is strictly concave, then $\Sg$ is a hyperplane perpendicular to $\ptl\Om$.
\end{corollary}

\begin{remark}
The previous corollary is valid for Gaussian half-spaces and slabs. The result also holds when $\Om=\rrn$. In particular, we deduce that any smooth, complete, orientable, $f$-stable hypersurface of finite weighted area in the Gauss space must be a hyperplane. This was previously established by McGonagle and Ross, see \cite[Cor.~2]{stable-Gauss}.
\end{remark}

\begin{remark}
\label{re:baal}
The use of test functions of the form $\escpr{\eta,N_p}$ for fixed $\eta\in\rrn$ in the stability condition can be found in several previous works. Sternberg and Zumbrun considered these functions to characterize stable cones in Euclidean balls \cite[Lem.~3.11]{sz2}, and for proving nonexistence of antipodally symmetric stable sets in strictly convex antipodally symmetric bounded domains \cite[Prop.~3.2]{sz2}. Barbosa, do Carmo and Eschenburg employed linear combinations of $\escpr{\eta,N_p}$ and $\escpr{\eta,p}$ to characterize the geodesic spheres as the unique compact stable hypersurfaces in Riemannian space forms of non-vanishing curvature \cite[Thm.~1.2]{bdce}. By means of functions like $\alpha+\escpr{\eta,N_p}$ with $\alpha$ constant, McGonagle and Ross have described smooth, complete, orientable hypersurfaces of constant $f$-mean curvature and finite index in $\rrn$ with Gaussian density, see \cite[Thm.~2]{stable-Gauss}. 
\end{remark}

\section{Characterization of weighted isoperimetric regions}
\label{sec:main}
\setcounter{equation}{0}

Let $\Om:=\rr^n\times (a.b)$ be an open half-space or slab in $\rrn$ endowed with a perturbation of the Gaussian density $\ga_c$ as in \eqref{eq:perturb}.  In this section we show that, when $\omega$ is smooth and concave, then the intersections with $\Om$ of half-spaces perpendicular to $\ptl\Om$ are the unique weighted minimizers in $\Om$, up to sets of volume zero. Our arguments also extend to $\Om=\rrn$, which in particular provides a new approach to describe weighted minimizers in Gauss space. Finally, we will combine optimal transport with the Gaussian isoperimetric inequality to discuss the general situation where $\omega$ is any concave and possibly non-smooth function.

\subsection{Half-spaces and slabs}
\label{subsec:halfslabs}
\noindent

In order to characterize the weighted minimizers in a half-space or slab $\Om$ we will use the stability property as a main tool. Indeed, since any weighted isoperimetric region is also a weighted stable set, we deduce from Theorem~\ref{th:stable} that the only candidates to minimize the weighted perimeter for fixed weighted volume are the intersections with $\Om$ of half-spaces parallel or perpendicular to $\ptl\Om$. When the perturbation of $\ga_c$ is strictly log-concave then we know from Lemma~\ref{lem:half-spaces} (iii) that parallel half-spaces do not minimize, since they are not weighted stable. In the next proposition we show that half-spaces perpendicular to $\ptl\Om$ are always isoperimetrically better than the parallel ones.

\begin{proposition}[Parallel half-spaces vs. perpendicular half-spaces]
\label{prop:comparison}
Consider an open half-space or slab $\Om:=\rr^n\times (a,b)$ endowed with the density $f=e^\psi$, where $\psi(p):=\omega(\pi(p))-c |p|^2$ for some concave function $\omega\in C^\infty(a,b)$ and some $c>0$. Then, a half-space perpendicular to $\ptl\Om$ has strictly less weighted perimeter in $\Om$ than a half-space parallel to $\ptl\Om$ with the same weighted volume.
\end{proposition}

\begin{proof}
The proof of the proposition relies on the fact that the profile function associated to a parallel one-parameter family of $f$-stationary hyperplanes satisfies a second order differential inequality, which becomes an equality when the family is perpendicular to $\ptl\Om$.

Let $T\sub\rrn$ be a linear hyperplane. We take an orthonormal basis $\{e_1,\ldots,e_n,\xi_T\}$, where $\{e_1,\ldots,e_{n}\}\sub T$ and $\xi_T$ is chosen so that $\xi_T=\ptl_t$ when $T$ is the horizontal hyperplane $x_{n+1}=0$. We identify a point $p\in\rrn$ with its coordinates $(z,t)=(z_1,\ldots,z_n,t)$ in the previous basis. Given a domain $U\subeq T$ we consider the family $\Sg_s:=U\times\{s\}$, where $-\infty\leq\alpha<s<\beta\leq +\infty$. Let $C_U:=U\times (\alpha,\beta)$, and denote by $E_s$ the cylinder $U\times (\alpha,s)$ with $s\in (\alpha,\beta)$. Recall that $\pi:\rrn\to\rr$ is the vertical projection.

By Lemma~\ref{lem:finitevolume} we know that $f$ is a bounded density of finite weighted volume, and that hyperplanes intersected with $\Om$ has finite weighted area. Thus, we can define the weighted volume and area functions $V_f(s):=V_f(E_s)$ and $A_f(s):=A_f(\Sg_s)=P_f(E_s,C_U)$. By using the change of variables formula and Fubini's theorem, we get
\begin{align}
\label{eq:gizmo1}
V_f(s)&=\int_\alpha^s A_f(t)\,dt,
\\
\label{eq:gizmo2} 
A_f(s)&=\int_U e^{\omega(\pi(z,s))-c |z|^2-c s^2}\,dz.
\end{align}
Now, we apply differentiation under the integral sign to obtain 
\begin{align*}
V_f'(s)&=A_f(s),
\\
A_f'(s)&=\int_U \{\omega'(\pi(z,s))\,\escpr{\xi_T,\ptl_t}-2c s\}\,
e^{\omega(\pi(z,s))-c |z|^2-c s^2}\,dz,
\end{align*} 
for any $s\in (\alpha,\beta)$. In particular, we deduce
\[
A_f'(s)=
\begin{cases}
(\omega'(s)-2c s)\,A_f(s), \ \text{ if } T \text{ is horizontal},
\\
-2c s\,A_f(s), \hspace{1.315cm} \text{ if } T \text{ is vertical}.
\end{cases}
\]
Let $F:[0,V_f(C_U)]\to\rr$ be the associated profile function given by $F(v):=A_f(V_f^{-1}(v))$. This function is continuous on $[0,V_f(C_U)]$  and $C^2$ on $(0,V_f(C_U))$. From the computations above and the concavity of $\omega$ it is straightforward to check that, when $T$ is horizontal, then $F''\leq -2c F^{-1}$ in $(0,V_f(C_U))$ with equality if and only if $\omega$ is an affine function. Moreover, we have $F''= -2c F^{-1}$ in $(0,V_f(C_U))$ when $T$ is vertical. 

The previous arguments can be applied when $\Sg_s$ is a family of hyperplanes in $\Om$ which are either parallel or perpendicular to $\ptl\Om$. If $F$ and $G$ denote the associated profile functions, then $F''\leq -2c F^{-1}$ in $(0,V_f(\Om))$ and $G''=-2c\,  G^{-1}$ in $(0,V_f(\Om))$. Note that $F(0)\geq 0$ and $F(V_f(\Om))\geq 0$ (indeed, these values are positive if $\Om$ is a slab and $f$ is smooth in $\overline{\Om}$). By taking into account that $G(0)=G(V_f(\Om))=0$, we conclude that $F\geq G$ in $(0,V_f(\Om))$ by Lemma~\ref{lem:analysis} below. Moreover, in case of equality for some $v\in (0,V_f(\Om))$, then we would have $F=G$ in $[0,V_f(\Om)]$. Thus $F''=-2cF^{-1}$ in $(0,V_f(\Om))$ and we would deduce that $\omega$ is an affine function. In particular $F(0)>0$ or $F(V_f(\Om))>0$, which contradicts that $F=G$. This shows that $F>G$ in $(0,V_f(\Om))$ and proves the claim.
\end{proof}

\begin{lemma}
\label{lem:analysis}
Let $F,G:[a,b]\to\rr$ be continuous functions of class $C^2$ on $(a,b)$, positive in $(a,b)$, and satisfying $F''\leq -2c F^{-1}$ and $G''=-2c\,G^{-1}$ in $(a,b)$ for some $c>0$. If $F(a)\geq G(a)$ and $F(b)\geq G(b)$, then $F\geq G$ in $[a,b]$. Moreover, if $F(t_0)=G(t_0)$ for some $t_0\in (a,b)$, then $F=G$ in $[a,b]$.  
\end{lemma}

\begin{proof}
Suppose that there is $t\in (a,b)$ where $F(t)<G(t)$. Then, the minimum of the function $F-G$ is achieved at some $t_1\in (a,b)$ for which $(F-G)(t_1)<0$. Thus, we get 
\[
0\leq (F-G)''(t_1)\leq -2c F(t_1)^{-1}+2c\,G(t_1)^{-1}<0.
\]
This shows that $F\geq G$ in $[a,b]$. Suppose now that $F(t_0)=G(t_0)$ for some $t_0\in (a,b)$, which in particular implies $F'(t_0)=G'(t_0)$. As the function $x\mapsto -2c/x$ is increasing for $x>0$, we can apply a classical comparison result for ordinary differential inequalities \cite[Thm.~19.1]{diffineq} to deduce that $F\leq G$ in $[t_0,b]$. By using the same result with the functions $F_1(t):=F(t_0-t)$ and $G_1(t):=G(t_0-t)$ on $[0,t_0-a]$ we conclude that $F\leq G$ in $[a,b]$. This fact together with inequality $F\geq G$ in $[a,b]$ proves the claim.
\end{proof}

We can now combine our main previous results to prove the following statement.

\begin{theorem}[Weighted minimizers in half-spaces and slabs]
\label{th:isoperimetric}
Consider an open half-space or slab $\Om:=\rr^n\times (a,b)$ endowed with the density $f=e^\psi$, where $\psi(p):=\omega(\pi(p))-c |p|^2$ for some concave function smooth in the closure of $(a,b)$ and some $c>0$. Then, weighted isoperimetric regions of any given volume exist in $\Om$ and, up to sets of volume zero, they are all intersections with $\Om$ of half-spaces perpendicular to $\ptl\Om$.
\end{theorem}

\begin{proof}
The concavity of $\omega$ implies by Lemma~\ref{lem:finitevolume} that $f$ is bounded and $V_f(\Om)<+\infty$. Hence, the existence of weighted minimizers is a consequence of Theorem~\ref{th:exist}. Let $E\sub\Om$ be a weighted isoperimetric region with $V_f(E)=v\in (0,V_f(\Om))$. From the regularity results in Theorem~\ref{th:reg} we know that the interior boundary $\overline{\ptl E\cap\Om}$ is a disjoint union $\Sg\cup\Sg_0$, where $\Sg$ is a smooth embedded hypersurface with (possibly empty) boundary $\ptl\Sg=\Sg\cap\ptl\Om$, and $\Sg_0$ is a closed singular set of Hausdorff dimension less than or equal to $n-7$. By Remark~\ref{re:isop} we can assume that $E$ is, up to a set of volume zero, an open set. From Theorem~\ref{th:nullcap} the hypersurface $\Sg$ has null weighted capacity. By the classification of weighted stable sets in Theorem~\ref{th:stable} we deduce that $E$ is the intersection with $\Om$ of a half-space with boundary parallel or perpendicular to $\ptl\Om$. Finally, we apply the comparison in Proposition~\ref{prop:comparison} to conclude that half-spaces perpendicular to $\ptl\Om$ are isoperimetrically better.
\end{proof}

\begin{remark}[Uniqueness]
\label{re:uniqueness}
The weighted minimizers in $\Om$ of a given volume are not unique. Since the density $f=e^\psi$ with $\psi(p):=\omega(\pi(p))-c |p|^2$ is invariant under any Euclidean vertical rotation $\phi$, then a set $E\sub\Om$ is a weighted minimizer of volume $v$ if and only if the same holds for $\phi(E)$. In fact, for any two half-spaces $E$ and $E'$ perpendicular to $\ptl\Om$ and with the same weighted volume in $\Om$, there is a rotation $\phi$ as above such that $\phi(E)=E'$. 
\end{remark}

An interesting particular case of Theorem~\ref{th:isoperimetric} is obtained when $\omega$ vanishes identically.

\begin{corollary}[Weighted minimizers in Gaussian half-spaces and slabs]
\label{th:isopGausshalf}
Let $\Om$ be an open half-spa\-ce or slab in $\rrn$ with Gaussian density $\ga_c(p):=e^{-c |p|^2}$, $c>0$. Then, weighted isoperimetric regions of any given volume exist in $\Om$ and they are all obtained, up to sets of volume zero, by intersecting $\Om$ with half-spaces perpendicular to $\ptl\Om$.
\end{corollary}

\begin{remark}
The isoperimetric problem in a Euclidean half-space $\Om$ with Gaussian density has been studied in previous works.  On the one hand, Lee \cite[Prop.~5.1]{gauss-halfspaces} used the same approximation argument as in \cite{st} to prove that half-spaces perpendicular to $\ptl\Om$ are weighted minimizers. Later on, Adams, Corwin, Davis, Lee and Visocchi~\cite[Thm.~3.1]{gauss-sectors} obtained uniqueness of minimizers when $0\in\ptl\Om$ by means of reflection across $\ptl\Om$ and the characterization of equality cases in the Gaussian isoperimetric inequality. Some basic properties for weighted isoperimetric regions in symmetric planar strips with Gaussian density were established by Lee \cite[Sect.~5]{gauss-halfspaces}. 
\end{remark}

\begin{remark}[Stability vs. isoperimetry]
The isoperimetric result in Theorem~\ref{th:isoperimetric} implies that half-spaces perpendicular to $\ptl\Om$ are weighted stable sets. However, not every weighted stable set is a weighted minimizer. For example, inside an open half-space or slab $\Om$ with Gaussian density, a half-space $E$ parallel to $\ptl\Om$ provides a weighted stable set by Lemma~\ref{lem:half-spaces} (iii). However, $E$ is not isoperimetric since a half-space perpendicular to $\ptl\Om$ of the same weighted volume is isoperimetrically better.
\end{remark}

The weighted stability of half-spaces perpendicular to $\ptl\Om$ allows us to apply the stability inequality in Proposition~\ref{prop:extended} (i) to deduce the following result. 

\begin{corollary}[A Poincar\'e type inequality]
\label{cor:poincare}
Consider an open half-space or slab $\Om:=\rr^n\times (a,b)$ endowed with the density $f=e^\psi$, where $\psi(p):=\omega(\pi(p))-c |p|^2$ for some concave function smooth in the closure of $(a,b)$ and some $c>0$. Let $E$ be the intersection with $\Om$ of a Euclidean half-space perpendicular to $\ptl\Om$. If we denote $\Sg:=\overline{\ptl E\cap\Om}$, then we have
\[
\int_\Sg |\nabla_\Sg u|^2\,da_f\geq 2c\int_\Sg u^2\,da_f
\]
for any function $u\in H^1(\Sg,da_f)$ with $\int_\Sg u\,da_f=0$.
\end{corollary}

\begin{remark}
\label{re:spectralgap}
A Poincar\'e type inequality provides a lower bound on the \emph{spectral gap} or \emph{Poincar\'e constant} of $\Sg$, which is the non-negative number defined as
\[
\lambda_f(\Sg):=\inf\left\{\frac{\int_\Sg |\nabla_\Sg u|^2\,da_f}{\int_\Sg u^2\,da_f}\,;\,u\in H^1(\Sg,da_f),\,u\neq 0,\,\int_\Sg u\,da_f=0\right\}.
\] 
It is well known that $\la_f(\Sg)=2c$ for a hyperplane $\Sg$ of $\rrn$ with Gaussian density $\ga_c$, see Ledoux~\cite{markov}.  Recently we learned from E.~Milman that an optimal transport argument as in \cite[Thm.~2.2]{milman-spectral} implies the lower bound $\la_f(\Sg)\geq 2c$ for any convex domain $\Sg$ of a hyperplane endowed with a log-concave perturbation of the Gaussian density. Our Corollary~\ref{cor:poincare} makes use of the stability condition to provide a different proof of inequality $\la_f(\Sg)\geq 2c$ in a particular situation. More general and unified approaches leading to spectral gap inequalities in compact weighted Riemannian manifolds can be found in \cite{bakry-qian} and the recent paper of Kolesnikov and Milman~\cite{milman-kolesnikov}.
\end{remark}

\subsection{The case $\Om=\rrn$. New proof of the Gaussian isoperimetric inequality}
\label{subsec:erreene}
\noindent

The arguments in the proof of Theorem~\ref{th:isoperimetric} can be adapted to solve the isoperimetric problem in $\rrn$ with a smooth log-concave perturbation of the Gaussian density as in \eqref{eq:perturb}.

\begin{theorem}[Weighted minimizers in $\rrn$]
\label{th:isoperimetric2}
Consider the density $f=e^\psi$, where $\psi(p):=\omega(\pi(p))-c |p|^2$ for some concave function $\omega\in C^\infty(\rr)$ and $c>0$. Then, weighted isoperimetric regions of any given volume exist for this density, and we have the following:
\begin{itemize}
\item[(i)] if $\omega$ is not an affine function, then any weighted minimizer is, up to a set of volume zero, a vertical half-space,
\item[(ii)] if $\omega$ is an affine function, then any weighted minimizer is, up to a set of volume zero, a half-space. Indeed, any half-space in $\rrn$ is a weighted minimizer.
\end{itemize}
\end{theorem}

\begin{proof}
Following the proof of Theorem~\ref{th:isoperimetric} and using the classification of weighted stable sets in Theorem~\ref{th:stableGausswhole}, we obtain that weighted minimizers of any volume exist, and they are all Euclidean half-spaces, up to sets of volume zero. Now we distinguish two cases.

\emph{Case 1.} If $\omega$ is not an affine function, then the boundary of any minimizer is either a horizontal or a vertical hyperplane. Let $F,G:[0,V_f(\rrn)]$ be the profile functions for these two families of candidates. From the computations in the proof of Proposition~\ref{prop:comparison}, we get $F''\leq -2c F^{-1}$ and $G''=-2c\,G^{-1}$ in $(0,V_f(\rrn))$. Since $F$ and $G$ coincides at the extremes, an application of Lemma~\ref{lem:analysis} yields $F\geq G$ in $[0,V_f(\rrn)]$. In fact, if $F(v_0)=G(v_0)$ for some $v_0\in (0,V_f(\rrn))$, then we would get $F=G$ in $[0,V_f(\rrn)]$, and so the function $F$ would satisfy $F''=-2c F^{-1}$. From here, we would deduce that $\omega''=0$ in $\rr$, a contradiction. So, we conclude that $F>G$ in $(0,V_f(\rrn))$, which proves the claim.

\emph{Case 2.} If $\omega$ is an affine function, then $\omega''=0$ and $\omega'$ is constant in $\rr$. By following the computations in Proposition~\ref{prop:comparison}, it is easy to check that the profile function $F$ associated to an arbitrary parallel family $\{\Sg_s\}_{s\in\rr}$ of hyperplanes satisfies $F''=-2c F^{-1}$ in $(0,V_f(\rrn))$ and vanishes at the extremes. From Lemma~\ref{lem:analysis}, we infer that $F$ does not depend on the family $\{\Sg_s\}_{s\in\rr}$. As a consequence, any half-space in $\rrn$ is a weighted minimizer. 
\end{proof}

As a particular case of Theorem~\ref{th:isoperimetric2} we provide a new proof of the isoperimetric property of Euclidean half-spaces in the Gauss space. The reader is referred to the beginning of the Introduction for an account of different proofs and applications of this important result.

\begin{corollary}[Weighted minimizers in the Gauss space]
Half-spaces uniquely solve the isoperimetric problem in $\rrn$ with Gaussian density $\ga_c(p):=e^{-c |p|^2}$, $c>0$.
\end{corollary}

\subsection{The general case}
\label{subsec:non-smooth}
\noindent

We now turn to discuss the isoperimetric question in a half-space or slab $\Om:=\rr^n\times (a,b)$ with a density $f$ as in \eqref{eq:perturb}, where $\omega$ is concave on $(a,b)$ \emph{and possibly non-smooth} in the closure of $(a,b)$. 

An important result in this direction was obtained by Brock, Chiacchio and Mercaldo~\cite[Thm.~2.1]{bcm}. They proved that, in $\Om:=\rr^n\times\rr^+$ with density $f(z,t):=t^m\,e^{-(|z|^2+t^2)/2}$, where $m\geq 0$, the intersections with $\Om$ of half-spaces perpendicular to $\ptl\Om$ uniquely minimize the weighted perimeter for fixed weighted volume. Note that the density $f$ is a particular case of \eqref{eq:perturb} with $c:=1/2$ and $\omega(t):=m \log(t)$. However, this statement does not follow from our Theorem~\ref{th:isoperimetric} since $\omega$ is not smooth on $\rr^+_0$ when $m>0$. For the proof, the authors employed optimal transport to define a $1$-Lispchitz map $T:\rrn\to\Om$ equaling the identity on $\rr^n\times\{0\}$ and pushing the Gaussian measure forward the weighted volume associated to $f$. From here, the isoperimetric comparison can be deduced from the Gaussian isoperimetric inequality. Recently, we heard from E.~Milman that this approach also leads to the isoperimetric property of half-spaces perpendicular to $\ptl\Om$ in the general case. The main fact is that a density as in \eqref{eq:perturb} is the product of an $n$-dimensional Gaussian factor and a $1$-dimensional log-concave Gaussian perturbation. Hence, we can reproduce the arguments in \cite[Thm.~2.1]{bcm} and \cite[Thm.~2.2]{milman-spectral} to prove the next result, which generalizes \cite[Thm.~2.1]{bcm}.

\begin{theorem}[The general case]
\label{th:nuevo}
Consider an open half-space or slab $\Om:=\rr^n\times (a,b)$ endowed with the density $f=e^\psi$, where $\psi(p):=\omega(\pi(p))-c |p|^2$ for some concave function $\omega:(a,b)\to\rr$ and $c>0$. Then, the intersections with $\Om$ of half-spaces perpendicular to $\ptl\Om$ are always weighted isoperimetric regions. Moreover, any weighted minimizer is the intersection with $\Om$ of a half-space parallel or perpendicular to $\ptl\Om$. If, in addition, we suppose that $\omega\in C^\infty(a,b)$, then half-spaces perpendicular to $\ptl\Om$ are the unique weighted minimizers.
\end{theorem}

\begin{proof}
We consider the one-dimensional measures $\mu_1:=\alpha\,e^{-cs^2}\,ds$ and $\mu_2:=\beta\,e^{\omega(t)-ct^2}\,dt$, where $\alpha$ and $\beta$ are constants so that $\mu_1(\rr)=\mu_2\big((a,b)\big)=1$. Let $\rho:\rr\to (a,b)$ be a non-decreasing and $C^1$ smooth function pushing $\mu_1$ forward $\mu_2$. This means that $\mu_2(D)=\mu_1(\rho^{-1}(D))$ for any $D\subeq (a,b)$. Then, a straightforward computation shows that
\[
\alpha\,\int_{-\infty}^se^{-cu^2}\,du=\beta\,\int_a^{\rho(s)}e^{\omega(u)-cu^2}\,du,\quad s\in\rr,
\]
from which we get
\begin{equation}
\label{eq:seso1}
\alpha\,e^{-c s^2}=\beta\,e^{\omega(\rho(s))-c\rho(s)^2}\rho'(s),\quad s\in\rr.
\end{equation}
On the other hand, the function $\rho$ coincides with the optimal transport Brenier map pushing $\mu_1$ forward $\mu_2$, see \cite[Thms.~1 and 2]{caffarelli}. This means that $\rho$ minimizes the cost $\int |s-\rho(s)|^2\,d\mu_1$ among all functions pushing $\mu_1$ forward $\mu_2$. Hence, we can invoke a result of Caffarelli~\cite[Thm.~11]{caffarelli}, see also Kim and Milman~\cite[Lem.~3.3]{contraction} for an extension to non-smooth measures, to deduce that $\rho$ is a $1$-Lispchitz function. In particular, we have 
\begin{equation}
\label{eq:seso2}
\rho'(s)\leq 1, \quad s\in\rr.
\end{equation}

Now, we define the $C^1$ map $T:\rr^n\times\rr\to\Om$ given by $T(z,s):=(z,\rho(s))$. It is clear that $T$ is a $1$-Lipschitz map. Moreover, as $\rho$ pushes $\mu_1$ forward $\mu_2$, then
\begin{equation}
\label{eq:seso3}
V_{\ga_c}\big(T^{-1}(E)\big)=\frac{\beta}{\alpha}\,V_f(E), \quad E\subeq\Om,
\end{equation}
where $\ga_c(z,s):=e^{-c|z|^2-cs^2}$ for any $(z,s)\in\rr^n\times\rr$. Let us see that the weighted perimeter $P_f(E,\Om)$ defined in \eqref{eq:wp2} satisfies
\begin{equation}
\label{eq:seso4}
P_f(E,\Om)\geq\frac{\alpha}{\beta}\,P_{\ga_c}(T^{-1}(E),\rrn), \quad E\sub\Om.
\end{equation}

Suppose first that $E$ is open and the boundary $\ptl E\cap\Om$ is a smooth hypersurface $\Sg$. We denote $\Sg':=T^{-1}(\Sg)$. For any $p=(z,s)\in\Sg'$, let $|\text{Jac}_{\Sg'}T|(z,s)$ be the squared root of the determinant of the matrix $(\escpr{u_i(T),u_j(T)})$, where $\{u_1,\ldots,u_n\}$ is an orthonormal basis of $T_p\Sg'$. Then, it is not difficult to check that
\[
|\text{Jac}_{\Sg'}T|(z,s)=\sqrt{1+\big(\rho'(s)^2-1\big)\,|\nabla_{\Sg'}\pi|^2(p)},
\]
where $\nabla_{\Sg'}\pi$ is the gradient of $\pi$ relative to $\Sg'$. By using \eqref{eq:seso2} and $|\nabla_{\Sg'}\pi|(p)\leq 1$, we get
\[
|\text{Jac}_{\Sg'}T|(z,s)\geq\rho'(s),
\]
and equality holds if and only if $\rho'(s)=1$ or the vertical vector $\ptl_s$ is tangent to $\Sg'$ at $p$. By taking into account Lemma~\ref{lem:percont} (ii), the previous inequality and equation \eqref{eq:seso1}, we obtain
\begin{align}
\label{eq:seso41}
P_f(E,\Om)&=\int_\Sg f\,da=\int_{\Sg'}(f\circ T)\,|\text{Jac}_{\Sg'}T|\,da
\\
\nonumber
&=\int_{\Sg'}e^{-c|z|^2}e^{\omega(\rho(s))-c\rho(s)^2}\,|\text{Jac}_{\Sg'}T|(z,s)\,da
\\
\nonumber
&\geq\frac{\alpha}{\beta}\,\int_{\Sg'}\ga_c(p)\,da=\frac{\alpha}{\beta}\,P_{\ga_c}(T^{-1}(E),\rrn),
\end{align}
where $da$ is the Euclidean element of area. Moreover, equality holds if and only if, for any point $p=(z,s)\in\Sg'$, we have $\rho'(s)=1$ or $\ptl_s\in T_p\Sg'$. This proves \eqref{eq:seso4} when $E$ is open and $\ptl E\cap\Om$ is smooth.

Suppose now that $E\sub\Om$ is any Borel set. From Lemma~\ref{lem:percont} (iii) we can find a sequence $\{E_k\}_{k\in\mathbb{N}}$ of open sets in $\Om$ such that $\ptl E_k\cap\Om$ is smooth, $\{E_k\}_{k\in\mathbb{N}}\to E$ in $L^1(\Om,dv_f)$ and $\lim_{k\to\infty}P_f(E_k,\Om)=P_f(E,\Om)$. By applying inequality \eqref{eq:seso4} to $E_k$ we get
\begin{equation}
\label{eq:seso6}
P_f(E_k,\Om)\geq\frac{\alpha}{\beta}\,P_{\ga_c}(T^{-1}(E_k),\rrn),\quad k\in\mathbb{N}.
\end{equation}
On the other hand, note that $\{T^{-1}(E_k)\}_{k\in\mathbb{N}}\to T^{-1}(E)$ in $L^1(\rrn,dv_{\ga_c})$. This is a consequence of \eqref{eq:seso3} since
\[
V_{\ga_c}\big(T^{-1}(E_k)\,\triangle\,T^{-1}(E)\big)=V_{\ga_c}\big(T^{-1}(E_k\,\triangle\,E)\big)=\frac{\beta}{\alpha}\,V_f(E_k\,\triangle\,E),
\]
and $\{E_k\}_{k\in\mathbb{N}}\to E$ in $L^1(\Om,dv_f)$. By taking limits in \eqref{eq:seso6} and using the lower semicontinuity of the weighted perimeter in Lemma~\ref{lem:percont} (i) we deduce inequality \eqref{eq:seso4}. 

Finally, take a Borel set $E\sub\Om$ and a half-space $H$ in $\rrn$ perpendicular to $\ptl\Om$ with $V_{\ga_c}(H)=V_{\ga_c}(T^{-1}(E))$. By combining the perimeter comparison in \eqref{eq:seso4} with the Gaussian isoperimetric inequality, we obtain
\begin{equation}
\label{eq:seso7}
P_f(E,\Om)\geq\frac{\alpha}{\beta}\,P_{\ga_c}(T^{-1}(E),\rrn)\geq\frac{\alpha}{\beta}\,P_{\ga_c}(H,\rrn)=P_f(H\cap\Om,\Om).
\end{equation}
The last equality comes from \eqref{eq:seso41} since $\ptl_s$ is always tangent to $\ptl H$. Note also that $V_f(H\cap\Om)=V_f(E)$ by \eqref{eq:seso3}. Hence, we have proved that half-spaces perpendicular to $\ptl\Om$ are weighted minimizers in $\Om$. Moreover, if $E$ is a weighted isoperimetric region in $\Om$, then equality holds in \eqref{eq:seso7}. In particular, $T^{-1}(E)$ is a Gaussian minimizer, and so it coincides, up to a set of volume zero, with a Euclidean half-space. In addition, we have also equality in \eqref{eq:seso41}, which implies that $T^{-1}(E)$ is either parallel or perpendicular to $\ptl\Om$ (otherwise $\rho:\rr\to (a,b)$ would be an affine function). If we further assume $\omega\in C^\infty(a,b)$ then we conclude from Proposition~\ref{prop:comparison} that perpendicular half-spaces are isoperimetrically better than parallel ones. 
\end{proof}

In the case $\Om=\rrn$ we can use the same arguments to prove the following result.

\begin{theorem}[The general case in $\rrn$]
\label{th:nuevo2}
Consider the density $f=e^\psi$, where $\psi(p):=\omega(\pi(p))-c |p|^2$ for some concave function $\omega:\rr\to\rr$ and $c>0$. Then, vertical half-spaces are always weighted minimizers. Moreover, any weighted minimizer must be a half-space.
\end{theorem}

\begin{remark}
\label{re:moreproofs}
The isoperimetric property of half-spaces perpendicular to $\ptl\Om$ can be established by a variety of ways. Here we gather some additional proofs. 

1. The first proof was explained to us by F.~Barthe. Indeed, as the function $f$ is a product density, then we can use tensorization properties as in Barthe and Maurey~\cite[Remark after Prop.~5]{bm} to obtain that the weighted isoperimetric profile $I_{\Om,f}$ defined in \eqref{eq:isopro} of the density $f$ (normalized to have unit weighted volume) is bounded from below by the profile $I_\ga$ of the Gaussian probability density in $\rrn$. By inspecting a half-space $H$ perpendicular to $\ptl\Om$ and using the Gaussian isoperimetric inequality we conclude that $I_{\Om,f}=I_\ga$ and $H$ is a weighted minimizer. 

2. The second proof relies on symmetrization with respect to a model measure as described by Ros~\cite[Thms.~22 and 23]{rosisoperimetric}. By taking into account that half-lines are weighted minimizers for any log-concave density on the real line, see \cite[Sect.~4]{rcbm} and the references therein, we deduce that $I_{\Om,f}\geq I_\ga$. Now, we can finish as in the previous proof.

3. The third proof is deduced from Theorem~\ref{th:isoperimetric} by approximation arguments. First, we consider parallel interior domains to $\Om$ in order to show the claim when $\omega\in C^\infty(a,b)$. Second, we construct a sequence of smooth concave functions $\omega_k$ for which the previous step can be applied. For a detailed development the reader is referred to the first version of the present paper, see [arXiv:1403.4510, Thm.~5.12].
\end{remark}

\providecommand{\bysame}{\leavevmode\hbox to3em{\hrulefill}\thinspace}
\providecommand{\MR}{\relax\ifhmode\unskip\space\fi MR }
\providecommand{\MRhref}[2]{%
  \href{http://www.ams.org/mathscinet-getitem?mr=#1}{#2}
}
\providecommand{\href}[2]{#2}


\begin{thebibliography}{10}

\bibitem{gauss-sectors}
E.~Adams, I.~Corwin, D.~Davis, M.~Lee, and R.~Visocchi, \emph{Isoperimetric
  regions in {G}auss sectors}, Rose-Hulman Und.~Math.~J. \textbf{8} (2007),
  no.~1.

\bibitem{ambrosio-bv}
L.~Ambrosio, \emph{Some fine properties of sets of finite perimeter in
  {A}hlfors regular metric measure spaces}, Adv. Math. \textbf{159} (2001),
  no.~1, 51--67. \MR{1823840 (2002b:31002)}

\bibitem{afp}
L.~Ambrosio, N.~Fusco, and D.~Pallara, \emph{Functions of bounded variation and
  free discontinuity problems}, Oxford Mathematical Monographs, The Clarendon
  Press Oxford University Press, New York, 2000. \MR{1857292 (2003a:49002)}

\bibitem{be}
D.~Bakry and M.~{\'E}mery, \emph{Diffusions hypercontractives}, S\'eminaire de
  probabilit\'es, {XIX}, 1983/84, Lecture Notes in Math., vol. 1123, Springer,
  Berlin, 1985, pp.~177--206. \MR{MR889476 (88j:60131)}

\bibitem{bl}
D.~Bakry and M.~Ledoux, \emph{L\'evy-{G}romov's isoperimetric inequality for an
  infinite-dimensional diffusion generator}, Invent. Math. \textbf{123} (1996),
  no.~2, 259--281. \MR{MR1374200 (97c:58162)}

\bibitem{bakry-qian}
D.~Bakry and Z.~Qian, \emph{Some new results on eigenvectors via dimension,
  diameter, and {R}icci curvature}, Adv. Math. \textbf{155} (2000), no.~1,
  98--153. \MR{1789850 (2002g:58048)}

\bibitem{baldi}
A.~Baldi, \emph{Weighted {BV} functions}, Houston J. Math. \textbf{27} (2001),
  no.~3, 683--705. \MR{1864805 (2002j:46045)}

\bibitem{bdce}
J.~L. Barbosa, M.~P. do~Carmo, and J.~Eschenburg, \emph{Stability of
  hypersurfaces of constant mean curvature in {R}iemannian manifolds}, Math. Z.
  \textbf{197} (1988), no.~1, 123--138. \MR{MR917854 (88m:53109)}

\bibitem{bbj}
M.~Barchiesi, A.~Brancolini, and V.~Julin, \emph{{S}harp dimension free
  quantitative estimates for the {G}aussian isoperimetric inequality},
  arXiv:1409.2106, September 2014.

\bibitem{chiara}
F.~Barthe, C.~Bianchini, and A.~Colesanti, \emph{Isoperimetry and stability of
  hyperplanes for product probability measures}, Ann. Mat. Pura Appl. (4)
  \textbf{192} (2013), no.~2, 165--190. \MR{3035134}

\bibitem{bm}
F.~Barthe and B.~Maurey, \emph{Some remarks on isoperimetry of {G}aussian
  type}, Ann. Inst. H. Poincar\'e Probab. Statist. \textbf{36} (2000), no.~4,
  419--434. \MR{MR1785389 (2001k:60055)}

\bibitem{bayle-thesis}
V.~Bayle, \emph{Propri\'et\'es de concavit\'e du profil isop\'erim\'etrique et
  applications}, Ph.D. thesis, Institut Fourier (Grenoble), 2003.

\bibitem{bayle-paper}
\bysame, \emph{A differential inequality for the isoperimetric profile}, Int.
  Math. Res. Not. (2004), no.~7, 311--342. \MR{2041647 (2005a:53050)}

\bibitem{bayle-rosales}
V.~Bayle and C.~Rosales, \emph{Some isoperimetric comparison theorems for
  convex bodies in {R}iemannian manifolds}, Indiana Univ. Math. J. \textbf{54}
  (2005), no.~5, 1371--1394. \MR{2177105 (2006f:53040)}

\bibitem{bellettini}
G.~Bellettini, G.~Bouchitt{{\'e}}, and I.~Fragal{{\`a}}, \emph{B{V} functions
  with respect to a measure and relaxation of metric integral functionals}, J.
  Convex Anal. \textbf{6} (1999), no.~2, 349--366. \MR{1736243 (2000k:49016)}

\bibitem{bobkov2}
S.~G. Bobkov, \emph{An isoperimetric inequality on the discrete cube, and an
  elementary proof of the isoperimetric inequality in {G}auss space}, Ann.
  Probab. \textbf{25} (1997), no.~1, 206--214. \MR{MR1428506 (98g:60033)}

\bibitem{bobkov-perturbation}
\bysame, \emph{Perturbations in the {G}aussian isoperimetric inequality}, J.
  Math. Sci. (N. Y.) \textbf{166} (2010), no.~3, 225--238, Problems in
  mathematical analysis. No. 45. \MR{2839030 (2012m:60050)}

\bibitem{bobkov-udre}
S.~G. Bobkov and K.~Udre, \emph{Characterization of {G}aussian measures in
  terms of the isoperimetric property of half-spaces}, Zap. Nauchn. Sem.
  S.-Peterburg. Otdel. Mat. Inst. Steklov. (POMI) \textbf{228} (1996),
  no.~Veroyatn. i Stat. 1, 31--38, 356. \MR{1449845 (98e:60056)}

\bibitem{borell}
C.~Borell, \emph{The {B}runn-{M}inkowski inequality in {G}auss space}, Invent.
  Math. \textbf{30} (1975), no.~2, 207--216. \MR{MR0399402 (53 \#3246)}

\bibitem{bcm}
F.~Brock, F.~Chiacchio, and A.~Mercaldo, \emph{A class of degenerate elliptic
  equations and a {D}ido's problem with respect to a measure}, J. Math. Anal.
  Appl. \textbf{348} (2008), no.~1, 356--365. \MR{2449353 (2010h:35146)}

\bibitem{cabre3}
X.~Cabr{{\'e}}, X.~Ros-Oton, and J.~Serra, \emph{{S}harp isoperimetric
  inequalities via the {ABP} method}, arXiv:1304.1724v3, April 2013.

\bibitem{caffarelli}
L.~A. Caffarelli, \emph{Monotonicity properties of optimal transportation and
  the {FKG} and related inequalities}, Comm. Math. Phys. \textbf{214} (2000),
  no.~3, 547--563. \MR{1800860 (2002c:60029)}

\bibitem{cvm}
A.~Ca{\~n}ete, M.~Miranda, and D.~Vittone, \emph{Some isoperimetric problems in
  planes with density}, J. Geom. Anal. \textbf{20} (2010), no.~2, 243--290.
  \MR{2579510 (2011a:49102)}

\bibitem{homostable}
A.~Ca{\~n}ete and C.~Rosales, \emph{Compact stable hypersurfaces with free
  boundary in convex solid cones with homogeneous densities}, Calc. Var.
  Partial Differential Equations \textbf{51} (2014), no.~3-4, 887--913.
  \MR{3268875}

\bibitem{ck}
E.~A. Carlen and C.~Kerce, \emph{On the cases of equality in {B}obkov's
  inequality and {G}aussian rearrangement}, Calc. Var. Partial Differential
  Equations \textbf{13} (2001), no.~1, 1--18. \MR{MR1854254 (2002f:26016)}

\bibitem{castro-rosales}
K.~Castro and C.~Rosales, \emph{Free boundary stable hypersurfaces in manifolds
  with density and rigidity results}, J.~Geom.~Phys. \textbf{79} (2014),
  14--28.

\bibitem{chambers}
G.~R. Chambers, \emph{Proof of the log-convex density conjecture},
  arXiv:1311.4012v2, December 2013.

\bibitem{cfmp}
A.~Cianchi, N.~Fusco, F.~Maggi, and A.~Pratelli, \emph{On the isoperimetric
  deficit in {G}auss space}, Amer. J. Math. \textbf{133} (2011), no.~1,
  131--186. \MR{2752937 (2012b:28007)}

\bibitem{pratelli-cinti}
E.~Cinti and A.~Pratelli, \emph{The $\varepsilon-\varepsilon^\beta$ property,
  the boundedness of isoperimetric sets in $\mathbb{R}^{N}$ with density, and
  some applications}, arXiv:1209.3624, September 2012.

\bibitem{calibrations}
T.~H. Doan, \emph{Some calibrated surfaces in manifolds with density}, J. Geom.
  Phys. \textbf{61} (2011), no.~8, 1625--1629. \MR{2802497 (2012e:53091)}

\bibitem{ehrhard}
A.~Ehrhard, \emph{Sym\'etrisation dans l'espace de {G}auss}, Math. Scand.
  \textbf{53} (1983), no.~2, 281--301. \MR{MR745081 (85f:60058)}

\bibitem{evans-gariepy}
L.~C. Evans and R.~F. Gariepy, \emph{Measure theory and fine properties of
  functions}, Studies in Advanced Mathematics, CRC Press, Boca Raton, FL, 1992.
  \MR{1158660 (93f:28001)}

\bibitem{figalli}
A.~Figalli and F.~Maggi, \emph{On the isoperimetric problem for radial
  log-convex densities}, Calc. Var. Partial Differential Equations \textbf{48}
  (2013), no.~3-4, 447--489. \MR{3116018}

\bibitem{fmp}
N.~Fusco, F.~Maggi, and A.~Pratelli, \emph{On the isoperimetric problem with
  respect to a mixed {E}uclidean-{G}aussian density}, J. Funct. Anal.
  \textbf{260} (2011), no.~12, 3678--3717. \MR{2781973 (2012c:49095)}

\bibitem{giusti}
E.~Giusti, \emph{Minimal surfaces and functions of bounded variation},
  Monographs in Mathematics, vol.~80, Birkh{\"a}user Verlag, Basel, 1984.
  \MR{775682 (87a:58041)}

\bibitem{go-ma-ta}
E.~Gonzalez, U.~Massari, and I.~Tamanini, \emph{On the regularity of boundaries
  of sets minimizing perimeter with a volume constraint}, Indiana Univ. Math.
  J. \textbf{32} (1983), no.~1, 25--37. \MR{684753 (84d:49043)}

\bibitem{grigorian}
A.~Grigor'yan, \emph{Isoperimetric inequalities and capacities on {R}iemannian
  manifolds}, The {M}az'ya anniversary collection, {V}ol. 1 ({R}ostock, 1998),
  Oper. Theory Adv. Appl., vol. 109, Birkh{\"a}user, Basel, 1999, pp.~139--153.
  \MR{1747869 (2002a:31009)}

\bibitem{gri-masa}
A.~Grigor'yan and J.~Masamune, \emph{Parabolicity and stochastic completeness
  of manifolds in terms of the {G}reen formula}, J. Math. Pures Appl. (9)
  \textbf{100} (2013), no.~5, 607--632. \MR{3115827}

\bibitem{gromov-GAFA}
M.~Gromov, \emph{Isoperimetry of waists and concentration of maps}, Geom.
  Funct. Anal. \textbf{13} (2003), no.~1, 178--215. \MR{MR1978494
  (2004m:53073)}

\bibitem{gruter}
M.~Gr{{\"u}}ter, \emph{Boundary regularity for solutions of a partitioning
  problem}, Arch. Rational Mech. Anal. \textbf{97} (1987), no.~3, 261--270.
  \MR{862549 (87k:49050)}

\bibitem{gruter-jost}
M.~Gr{{\"u}}ter and J.~Jost, \emph{Allard type regularity results for varifolds
  with free boundaries}, Ann. Scuola Norm. Sup. Pisa Cl. Sci. (4) \textbf{13}
  (1986), no.~1, 129--169. \MR{863638 (89d:49048)}

\bibitem{howe}
S.~Howe, \emph{{T}he log-convex density conjecture and vertical surface area in
  warped products}, arXiv:1107.4402, July 2011.

\bibitem{contraction}
Y.-H. Kim and E.~Milman, \emph{A generalization of {C}affarelli's contraction
  theorem via (reverse) heat flow}, Math. Ann. \textbf{354} (2012), no.~3,
  827--862. \MR{2983070}

\bibitem{milman-kolesnikov}
A.~V. Kolesnikov and E.~Milman, \emph{Poincar\'e and {B}runn-{M}inkowski
  inequalities on weighted {R}iemannian manifolds with boundary},
  arXiv:1310.2526v4, September 2014.

\bibitem{ledoux-survey}
M.~Ledoux, \emph{Isoperimetry and {G}aussian analysis}, Lectures on probability
  theory and statistics ({S}aint-{F}lour, 1994), Lecture Notes in Math., vol.
  1648, Springer, Berlin, 1996, pp.~165--294. \MR{1600888 (99h:60002)}

\bibitem{ledoux-gaussian}
\bysame, \emph{A short proof of the {G}aussian isoperimetric inequality}, High
  dimensional probability ({O}berwolfach, 1996), Progr. Probab., vol.~43,
  Birkh{\"a}user, Basel, 1998, pp.~229--232. \MR{1652328 (99j:60027)}

\bibitem{markov}
\bysame, \emph{The geometry of {M}arkov diffusion generators}, Ann. Fac. Sci.
  Toulouse Math. (6) \textbf{9} (2000), no.~2, 305--366, Probability theory.
  \MR{1813804 (2002a:58045)}

\bibitem{gauss-halfspaces}
M.~Lee, \emph{Isoperimetric regions in surfaces and in surfaces with density},
  Rose-Hulman Und.~Math.~J. \textbf{7} (2006), no.~2.

\bibitem{lr}
G.~P. Leonardi and S.~Rigot, \emph{Isoperimetric sets on {C}arnot groups},
  Houston J. Math. \textbf{29} (2003), no.~3, 609--637 (electronic).
  \MR{MR2000099 (2004d:28008)}

\bibitem{lich1}
A.~Lichnerowicz, \emph{Vari{\'e}t{\'e}s riemanniennes {\`a} tenseur {C} non
  n{\'e}gatif}, C. R. Acad. Sci. Paris S{\'e}r. A-B \textbf{271} (1970),
  A650--A653. \MR{0268812 (42 \#3709)}

\bibitem{lich2}
\bysame, \emph{Vari{\'e}t{\'e}s k{\"a}hl{\'e}riennes {\`a} premi{\`e}re classe
  de {C}hern non negative et vari{\'e}t{\'e}s riemanniennes {\`a} courbure de
  {R}icci g{\'e}n{\'e}ralis{\'e}e non negative}, J. Differential Geom.
  \textbf{6} (1971/72), 47--94. \MR{0300228 (45 \#9274)}

\bibitem{maggi}
F.~Maggi, \emph{Sets of finite perimeter and geometric variational problems},
  Cambridge Studies in Advanced Mathematics, vol. 135, Cambridge University
  Press, Cambridge, 2012, An introduction to geometric measure theory.
  \MR{2976521}

\bibitem{stable-Gauss}
M.~McGonagle and J.~Ross, \emph{The hyperplane is the only stable, smooth
  solution to the isoperimetric problem in {G}aussian space}, arXiv:1307.7088,
  July 2013.

\bibitem{milman}
E.~Milman, \emph{Sharp isoperimetric inequalities and model spaces for
  curvature-dimension-diameter condition}, to appear in J.~Eur.~Math.~Soc.,
  arXiv:1108.4609v3.

\bibitem{milman-spectral}
\bysame, \emph{A proof of {B}obkov's spectral bound for convex domains via
  {G}aussian fitting and free energy estimation}, Analysis and geometry of
  metric measure spaces, CRM Proc. Lecture Notes, vol.~56, Amer. Math. Soc.,
  Providence, RI, 2013, pp.~181--196. \MR{3060503}

\bibitem{milman-rotem}
E.~Milman and L.~Rotem, \emph{Complemented {B}runn-{M}inkowski inequalities and
  isoperimetry for homogeneous and non-homogeneous measures}, Adv. Math.
  \textbf{262} (2014), 867--908. \MR{3228444}

\bibitem{miranda-bv}
M.~Miranda, \emph{Functions of bounded variation on ``good'' metric spaces}, J.
  Math. Pures Appl. (9) \textbf{82} (2003), no.~8, 975--1004. \MR{2005202
  (2004k:46038)}

\bibitem{morgan-reg}
F.~Morgan, \emph{Regularity of isoperimetric hypersurfaces in {R}iemannian
  manifolds}, Trans. Amer. Math. Soc. \textbf{355} (2003), no.~12, 5041--5052.
  \MR{1997594 (2004j:49066)}

\bibitem{morgandensity}
\bysame, \emph{Manifolds with density}, Notices Amer. Math. Soc. \textbf{52}
  (2005), no.~8, 853--858. \MR{MR2161354 (2006g:53044)}

\bibitem{gmt}
\bysame, \emph{Geometric measure theory. {A} beginner's guide}, fourth ed.,
  Elsevier/Academic Press, Amsterdam, 2009. \MR{2455580 (2009i:49001)}

\bibitem{lcdc}
\bysame, \emph{The log-convex density conjecture}, Concentration, functional
  inequalities and isoperimetry, Contemp. Math., vol. 545, Amer. Math. Soc.,
  Providence, RI, 2011, pp.~209--211. \MR{2858534}

\bibitem{morgan-johnson}
F.~Morgan and D.~L. Johnson, \emph{Some sharp isoperimetric theorems for
  {R}iemannian manifolds}, Indiana Univ. Math. J. \textbf{49} (2000), no.~3,
  1017--1041. \MR{1803220 (2002e:53043)}

\bibitem{morgan-pratelli}
F.~Morgan and A.~Pratelli, \emph{Existence of isoperimetric regions in
  {$\Bbb{R}^n$} with density}, Ann. Global Anal. Geom. \textbf{43} (2013),
  no.~4, 331--365. \MR{3038539}

\bibitem{morgan-rit}
F.~Morgan and M.~Ritor{{\'e}}, \emph{Isoperimetric regions in cones}, Trans.
  Amer. Math. Soc. \textbf{354} (2002), no.~6, 2327--2339. \MR{1885654
  (2003a:53089)}

\bibitem{cones}
M.~Ritor{{\'e}} and C.~Rosales, \emph{Existence and characterization of regions
  minimizing perimeter under a volume constraint inside {E}uclidean cones},
  Trans. Amer. Math. Soc. \textbf{356} (2004), no.~11, 4601--4622. \MR{2067135
  (2005g:49076)}

\bibitem{rosisoperimetric}
A.~Ros, \emph{The isoperimetric problem}, Global theory of minimal surfaces,
  Clay Math. Proc., vol.~2, Amer. Math. Soc., Providence, RI, 2005,
  pp.~175--209. \MR{MR2167260 (2006e:53023)}

\bibitem{rcbm}
C.~Rosales, A.~Ca{\~n}ete, V.~Bayle, and F.~Morgan, \emph{On the isoperimetric
  problem in {E}uclidean space with density}, Calc. Var. Partial Differential
  Equations \textbf{31} (2008), no.~1, 27--46. \MR{2342613 (2008m:49212)}

\bibitem{simon}
L.~Simon, \emph{Lectures on geometric measure theory}, Proceedings of the
  Centre for Mathematical Analysis, Australian National University, vol.~3,
  Australian National University Centre for Mathematical Analysis, Canberra,
  1983. \MR{756417 (87a:49001)}

\bibitem{sz2}
P.~Sternberg and K.~Zumbrun, \emph{A {P}oincar{\'e} inequality with
  applications to volume-constrained area-minimizing surfaces}, J. Reine Angew.
  Math. \textbf{503} (1998), 63--85. \MR{1650327 (99g:58028)}

\bibitem{sz}
\bysame, \emph{On the connectivity of boundaries of sets minimizing perimeter
  subject to a volume constraint}, Comm. Anal. Geom. \textbf{7} (1999), no.~1,
  199--220. \MR{1674097 (2000d:49062)}

\bibitem{st}
V.~N. Sudakov and B.~S. Tirel'son, \emph{Extremal properties of half-spaces for
  spherically invariant measures}, Zap. Nau\v cn. Sem. Leningrad. Otdel. Mat.
  Inst. Steklov. (LOMI) \textbf{41} (1974), 14--24, 165, Problems in the theory
  of probability distributions, II. \MR{MR0365680 (51 \#1932)}

\bibitem{diffineq}
J.~Szarski, \emph{Differential inequalities}, Monografie Matematyczne, Tom 43,
  Pa{\'n}stwowe Wydawnictwo Naukowe, Warsaw, 1965. \MR{0190409 (32 \#7822)}

\bibitem{troyanov}
M.~Troyanov, \emph{Parabolicity of manifolds}, Siberian Adv. Math. \textbf{9}
  (1999), no.~4, 125--150. \MR{1749853 (2001e:31013)}

\end{thebibliography}
\end{document}